\documentclass[a4paper,11pt,leqno]{article}

\usepackage[latin1]{inputenc}
\usepackage[T1]{fontenc} 
\usepackage[english]{babel}
\usepackage{verbatim}
\usepackage{calligra}
\usepackage{enumerate}
\usepackage{dsfont}
\usepackage{amsfonts}
\usepackage{amsmath}
\usepackage{amsthm}
\usepackage{amssymb}
\usepackage{mathtools}
\usepackage{mathrsfs}
\usepackage{color}
\usepackage{graphicx}
\usepackage{bm}

\usepackage{tikz}
\usepackage[retainorgcmds]{IEEEtrantools}

\usepackage{graphicx}		
\usepackage[hypcap=false]{caption}

\DeclareMathOperator*{\tend}{\longrightarrow}

\DeclareMathOperator*{\D}{\rm{div}}

\DeclareMathOperator*{\wtend}{\rightharpoonup}

\theoremstyle{definition}
\newtheorem{defi}{Definition}[section]

\newtheorem{rmk}[defi]{Remark}

\theoremstyle{plane}
\newtheorem{thm}[defi]{Theorem}
\newtheorem{prop}[defi]{Proposition}
\newtheorem{cor}[defi]{Corollary}
\newtheorem{lemma}[defi]{Lemma}

\newcommand{\tbf}{\textbf}

\newcommand{\tsl}{\textsl}

\newcommand{\mbb}{\mathbb}

\newcommand{\mc}{\mathcal}
\newcommand{\mf}{\mathfrak}

\newcommand{\veps}{\varepsilon}
\newcommand{\eps}{\varepsilon}
\newcommand{\what}{\widehat}

\newcommand{\vphi}{\varphi}

\newcommand{\ra}{\rightarrow}

\newcommand{\g}{\gamma}
\renewcommand{\k}{\kappa}
\newcommand{\s}{\sigma}
\renewcommand{\t}{\tau}
\newcommand{\z}{\zeta}

\newcommand{\de}{\delta}

\newcommand{\R}{\mathbb{R}}
\newcommand{\Q}{\mathbb{Q}}

\newcommand{\N}{\mathbb{N}}
\newcommand{\Z}{\mathbb{Z}}
\newcommand{\T}{\mathbb{T}}
\renewcommand{\P}{\mathbb{P}}

\renewcommand{\div}{{\rm div}\,}

\newcommand{\curl}{{\rm curl}\,}

\newcommand{\Id}{{\rm Id}\,}
\newcommand{\Supp}{{\rm Supp}\,}

\newcommand{\du}{\delta u_\varepsilon}
\newcommand{\dr}{\delta r_\varepsilon}
\newcommand{\db}{\delta b_\varepsilon}
\newcommand{\dx}{ \, {\rm d} x}
\newcommand{\dt}{ \, {\rm d} t}
\newcommand{\B}{B^s_{p, r}}
\newcommand{\al}{\alpha}
\newcommand{\bt}{\beta}

\allowdisplaybreaks

\def\d{\partial}
\def\div{{\rm div}\,}

\begin{document}

\newcommand{\dimitri}[1]{\textcolor{red}{[***DC: #1 ***]}}
\newcommand{\fra}[1]{\textcolor{blue}{[***FF: #1 ***]}}

\title{\textsc{\Large{\textbf{Rigorous derivation and well-posedness of a \\
quasi-homogeneous ideal MHD system}}}}

\author{\normalsize\textsl{Dimitri Cobb}$\,^1\qquad$ and $\qquad$
\textsl{Francesco Fanelli}$\,^{2}$ \vspace{.5cm} \\
\footnotesize{$\,^{1,} \,^2\;$ \textsc{Universit\'e de Lyon, Universit\'e Claude Bernard Lyon 1}}  \vspace{.1cm} \\
{\footnotesize \it Institut Camille Jordan -- UMR 5208}  \vspace{.1cm}\\
{\footnotesize 43 blvd. du 11 novembre 1918, F-69622 Villeurbanne cedex, FRANCE} \vspace{.2cm} \\
\footnotesize{$\,^{1}\;$\ttfamily{cobb@math.univ-lyon1.fr}}, $\;$
\footnotesize{$\,^{2}\;$\ttfamily{fanelli@math.univ-lyon1.fr}}
\vspace{.2cm}
}

\date\today

\maketitle

\subsubsection*{Abstract}
{\footnotesize The goal of this paper is twofold. On the one hand, we introduce a quasi-homogeneous version of the classical ideal MHD system and study its well-posedness
in critical Besov spaces $B^s_{p,r}(\R^d)$, $d\geq2$, with $1<p<+\infty$ and under the Lipschitz condition $s>1+d/p$ and $r\in[1,+\infty]$, or $s=1+d/p$ and $r=1$.
A key ingredient is the reformulation of the system \tsl{via} the so-called Els\"asser variables.
On the other hand, we give a rigorous justification of quasi-homogeneous MHD models, both in the ideal and in the dissipative cases: when $d=2$, we will derive them
from a non-homogeneous incompressible MHD system with Coriolis force, in the regime of low Rossby number and for small density variations around a constant state.
Our method of proof relies on a relative entropy inequality for the primitive system, and yields
precise rates of convergence, depending on the size of the initial data, on the order of the Rossby number and on the regularity of the viscosity and resistivity coefficients.
}

\paragraph*{\small 2010 Mathematics Subject Classification:}{\footnotesize 35Q35 
(primary);
76W05, 
35B40, 
76B03, 
35L60 
(secondary).}

\paragraph*{\small Keywords: }{\footnotesize quasi-homogeneous ideal MHD; Els\"asser variables; critical regularity; singular perturbation; low Rossby number; relative entropy inequality.}


\section{Introduction} \label{s:intro}

Magnetohydrodynamic equations are used to model conducting fluids which are subject to a self-generated magnetic field. This can be used to describe the dynamics of molten metal,
certain types of plasmas or electrolytes. More precisely, MHD equations are an accurate description of a fluid provided that:
\begin{enumerate}[(i)]
\item the fluid is non-relativistic, \tsl{i.e.} the characteristic speed of the fluid is small when compared to the speed of light; in this regime, the electrostatic approximation is valid;
\item the fluid is highly collisional, so that the hydrodynamic approximation holds.
\end{enumerate}
We refer to Section 2.6 of \cite{Bellan} for additional details, as well as for a (formal) derivation of the MHD system from the Vlasov equations.

The main focus of this article is the investigation of some mathematical questions related to the following \emph{quasi-homogeneous incompressible ideal MHD system}: 
\begin{equation}\label{i_eq:MHD-I}
\begin{cases}
\partial_t R + \D \big( RU \big) = 0\\[1ex]
\partial_t U + \D(U \otimes U - B \otimes B) + R \mathfrak{C}U + \nabla \left( \Pi + \dfrac{1}{2} |B|^2 \right) = 0\\[1ex]
\partial_t B + \D(U \otimes B - B \otimes U) = 0\\[1ex]
\D(U) = \div(B) = 0.
\end{cases}
\end{equation}
The previous system is set in $\R_+\times\R^d$, $d\geq2$, but the whole study can be performed also in the case of the flat torus $\T^d$ with minor modifications.
The unknowns consist of the velocity field $U$, the hydrodynamic pressure $\Pi$ and the magnetic field $B$ of the fluid, as well as a scalar function $R$ which plays the role of a density perturbation function.
In the above, $\mathfrak{C} \in\mc  M_d(\mathbb{R})$ is a constant $d \times d$ matrix.

The MHD system \eqref{i_eq:MHD-I} is appropriate for ideal incompressible magnetofluids which are subject to forces that depend linearly on $U$
and on small variations $R$ in the fluid density. This type of system is especially relevant for fluids evolving in a rotating frame of coordinates, such as geophysical or stellar fluids.
As a matter of fact, one of the main goals of this paper is to show how quasi-homogeneous MHD systems (both in the dissipative and ideal cases) arise naturally in the study of fast rotating fluids.
In particular, by using a relative entropy method, we prove that (regular enough) solutions of \eqref{i_eq:MHD-I} are in fact limits of solutions of a fully non-homogeneous MHD system undergoing fast rotation,
and we give an explicit speed of convergence. 
However, such an asymptotic result rests on a well-posedness theory for our quasi-homogeneous ideal MHD system: this is the second main goal of the present article.
We will study existence and uniqueness questions about solutions to \eqref{i_eq:MHD-I} in the framework of \emph{critical} Besov spaces, embedded in the set of globally Lipschitz functions.

Let us comment more in detail about both issues, well-posedness and rigorous derivation of equations \eqref{i_eq:MHD-I}, in Subsections \ref{ss:i_ideal} and \ref{ss:i_rig} respectively.

\subsection{Remarks concerning the ideal MHD} \label{ss:i_ideal}

While it is tempting to consider system \eqref{i_eq:MHD-I} a mere variation of the Euler equations, ideal MHD has a number of intriguing features that are entirely its own,
such as additional symmetries and conserved quantities (magnetic and cross helicities, for example).
Even in the case $R\equiv0$, the magnetic field is not a trivial addition to the problem, as is shown by the existence
of interesting and physically relevant static solutions (with $U \equiv 0$), which present their own mathematical challenges (see \textsl{e.g.} Section 9.8.3 of \cite{Bellan}
concerning the Grad-Shafranov equation). 

Even from a mathematical perspective, the added complexity is not simply computational, as can be seen by the fact that global existence of regular solutions remains unknown for ideal MHD,
even in two dimensions of space. This is, of course, in stark contrast with the situation of $2$-D homogeneous (\tsl{i.e.} with $R\equiv0$) non-conducting (\tsl{i.e.} with $B \equiv 0$) fluids.

\subsubsection{Previous well-posedness results}

Despite these difficulties, a number of results have been obtained concerning local in time well-posedness of the classical ideal MHD, \tsl{i.e.} system \eqref{i_eq:MHD-I} with $R\equiv0$. 

Firstly, well-posedness has been established in \cite{Sch} in Sobolev spaces $W^{m, 2} = H^m$ when the fluid is confined to a bounded domain with perfectly conducting boundary,
provided that the regularity exponent meets the condition $m > 1+d/2$. Notice that such a constraint on the regularity index is classical in hyperbolic problems.
That result was later extended in \cite{Secc} to all Lebesgue exponents $p > 1$,
as long as the regularity exponent of the Sobolev space $W^{m, p}$ satisfies $m > 1 + d/p$.
More recently, in order to reach the critical exponent $m_* = 1 + d/p$, well-posedness has been proved in the framework of critical Besov spaces \cite{MY}, Triebel-Lizorkin spaces \cite{C-M-Z} as well as Besov-Morrey spaces \cite{XZ}.

For these local solutions, several blow-up criteria have been obtained, much in the spirit of that of Beale-Kato-Majda \cite{BKM} for the Euler equations. More precisely, as proved in \cite{CKS}
assuming $R = 0$ in \eqref{i_eq:MHD-I}, a regular solution $(U, B)$ of the ideal MHD system defined for times $0 \leq t < T$ can be extended beyond $T$ if and only if
\begin{equation*}
\int_0^T \Big\{ \| \curl (U) \|_{L^\infty} + \| \curl(B) \|_{L^\infty} \Big\} \dt < +\infty\,.
\end{equation*}
Several improvements have been achieved in this direction: \cite{LZ} allows the $L^\infty$ norms to be replaced by the Besov norms $\dot{B}^0_{\infty, \infty}$, and this result was in turn improved by \cite{CCM}. The $L^\infty$ norms can also be replaced by the Triebel-Lizorkin norms $\dot{F}_{\infty, \infty}^0$, as shown in \cite{C-M-Z}

Finally, we point out that some results exist in the direction of paradoxical solutions. These solutions have finite energy (they are in fact $H^\beta$, for some $\beta > 0$)
but do not preserve magnetic helicity \cite{BBV}. We also refer to \cite{FLS} for the existence of bounded solutions that violate conservation of energy and cross helicity. Generally speaking, this means that ``reasonable'' solutions must be of some regularity.

\subsubsection{The Els\"asser variables}

The local well-posedness results in \cite{Sch}, \cite{Secc}, \cite{MY}, \cite{C-M-Z} and \cite{XZ} rely on the fact that the ideal MHD equations can be symmetrised by the following change of variables:
\begin{equation*}
\al = U + B \qquad \text{and} \qquad \bt = U - B.
\end{equation*}
These quantities are called \textit{Els\"asser variables}, and have been used by physicists\footnote{The Els\"asser variables, often called ``characteristic variables'',
are usually denoted by $z_\pm$ is the physical literature, as well as in some mathematical works.}
since the 1950s, mainly to study the propagation of magnetohydrodynamic waves in the linearised equations. Their main advantage is that they solve a much more symmetric system, namely
\begin{equation}\label{i:Elsasser}
\begin{cases}
\partial_t \al + (\bt \cdot \nabla) \al + \dfrac{1}{2}R \mathfrak{C} (\al + \bt) = - \nabla\pi\\[1ex]
\partial_t \bt + (\al \cdot \nabla) \bt + \dfrac{1}{2}R \mathfrak{C} (\al + \bt) = - \nabla\pi\\[1ex]
\D(\al) = \D(\bt) = 0\,,
\end{cases}
\end{equation}
where, here and throughout this paper, we have defined the magnetohydrodynamic pressure
\begin{equation} \label{eq:p-P}
\pi\,:=\,\Pi + \dfrac{1}{2} |B|^2\,.
\end{equation}
In addition to being more symmetric, the equations in \eqref{i:Elsasser} are transport equations, and this fact is instrumental for proving local well-posedness, although the Els\"asser variables do not appear explicitly in \cite{MY} and \cite{XZ} (instead, these papers feature compositions by the flows of $\al$ and $\bt$).

All we have done to obtain \eqref{i:Elsasser} is a simple change of variables; however, several interesting remarks can be made concerning the Els\"asser formulation.
As explained in \cite{Sch} and below (and much more thoroughly in \cite{Cobb-F_Els}), there is no need to make the \textsl{a priori} assumption that the gradient terms appearing in the right-hand
sides of \eqref{i:Elsasser} are equal (see \eqref{eq:MHDab} below). Instead, they can be assumed to be \emph{independent} Lagrange multipliers associated with the two (independent)
divergence-free constraints $\D(\al) = \D(\bt) = 0$. The fact that they end up being equal can be seen as an \textsl{a fortiori} consequence of the structure of the equations. 

In this article, we make intensive use of this fact. By treating the gradient terms in \eqref{i:Elsasser} as independent, we are able to project \eqref{i:Elsasser} on the space of
divergence-free functions with the Leray projector $\mathbb{P} = \text{Id}+ \nabla (- \Delta)^{-1}\D$. This  highlights the fact that, fundamentally, \eqref{i_eq:MHD-I} is a system of transport equations. 
We refer to Subsection \ref{ss:Elsasser} below for details.

The identification of the underlying transport structure \eqref{i:Elsasser} allows us to leave the strict setting of quasi-linear hyperbolic systems and deal with Lebesgue exponents $p \neq 2$.
Owing to this, we perform our study of the well-posedness of equations \eqref{i_eq:MHD-I} in the framework of Besov spaces $\B(\R^d)$, whose regularity indices satisfy the Lipschitz condition
\begin{equation}\label{i_eq:Lip}
s>1+\frac{d}{p}\quad \text{ and }\quad r\in [1,+\infty]\,,\qquad\qquad\text{ or }\qquad\qquad s=1+\frac{d}{p} \quad \text{ and }\quad r=1\,.
\end{equation}
However, using the Leray projector introduces problems if we want to handle the case where the Lebesgue exponent is $p = +\infty$ or $p = 1$, as $\mathbb{P}$ is not bounded on $L^1$,
and is even ill-defined as a Fourier multiplier on $L^\infty$.
Therefore, in all that follows, we restrict our attention to the case
\begin{equation}\label{i_eq:p}
1 < p < +\infty\,.
\end{equation}

The case $p=+\infty$ is left aside, and will be the matter of a forthcoming paper \cite{Cobb-F_Els}. There are several reasons for that.
First of all, as already said, the Leray projector $\P$ is not well-defined in $L^\infty$-type spaces; therefore, treating the endpoint case $p=+\infty$ requires a different
approach than the one used here. In addition, in absence of some global integrability conditions, it is not clear for us that the original formulation \eqref{i_eq:MHD-I} and the Els\"asser formulation \eqref{i:Elsasser} are equivalent. In the same way, without assuming some integrability for the solutions, the equivalence between each of those formulations and the corresponding system obtained after application of the operator $\P$ also seems questionable. In \cite{Cobb-F_Els} we will explore that issue in great detail, and show sufficient conditions in order for those equivalences to hold.

In passing, we mention also that, when $p=+\infty$ and in the case of space dimension $d=2$, it is possible to show an improved lower bound on the lifespan of the solutions
to the quasi-homogeneous ideal MHD system \eqref{i_eq:MHD-I}, so also for the classical ideal MHD system (when $R\equiv0$). This improved lower bound implies an ``asymptotically global''
well-posedness result: roughly speaking, if both the initial data $R_0$ and $B_0$ are of order $\veps>0$, then the lifespan $T_\veps$ of the corresponding solution satisfies
\begin{equation*}
T_\veps\, \tend_{\veps \rightarrow 0^+}\, +\infty\,.
\end{equation*}
Notice that this property does not follow from standard lower bounds coming from classical hyperbolic theory.

To conclude this part, we remark that, in a purely $L^\infty$ framework (that is, without any additional integrability condition), uniqueness of solutions to \eqref{i_eq:MHD-I}
is \emph{not} true in general, while uniqueness in the same setting do hold after projection onto the space of divergence-free vector fields (see \tsl{e.g.} \cite{MY}).
This issue has to be related to the failure of the equivalence between the two formulations, which we have mentioned above.

\subsection{Quasi-homogeneous MHD and fast rotating fluids} \label{ss:i_rig}

The second part of this article is dedicated to the rigorous derivation of the quasi-homogeneous ideal MHD system \eqref{i_eq:MHD-I}. The convergence will be shown
in the case of space dimension $d=2$, from a non-homogeneous fast rotating problem, namely
\begin{equation}\label{i:MHDnHom}
\begin{cases}
\partial_t \rho + \D(\rho u) = 0\\[1ex]
\partial_t (\rho \, u) + \D (\rho \, u \otimes u - b \otimes b) + \nabla \left( \dfrac{1}{\veps}  \Pi + \dfrac{1}{2} |b|^2 \right) + \dfrac{1}{\veps} \rho u^\perp = 
h(\veps) \D \big( \nu(\rho) \nabla u \big)\\[1ex]
\partial_t b + \D (u \otimes b - b \otimes u) = h(\veps) \curl \big( \mu(\rho) \curl (b) \big)\\[1ex]
\D(u) = \D(b) = 0\,.
\end{cases}
\end{equation}
This problem is set in the plane $\mathbb{R}^2$, or in the flat torus $\T^2$. The quantities $h(\veps)\nu(\rho)$ and $h(\veps)\mu(\rho)$ are the viscosity and resistivity of the fluid.
The fluid evolves in a rotating frame of coordinates, whence the presence of a Coriolis force term
\begin{equation}\label{i:eqCor}
F_C = \frac{1}{\veps} \rho u^\perp = \frac{1}{Ro} \rho u^\perp 
\end{equation}
associated to the rotation speed $1/Ro$, where $Ro$ is the Rossby number of the fluid. This type of study is especially relevant for geophysical or stellar fluids, which lie at the surface
of a rotating celestial body. In these problems, the Coriolis force dominates the dynamics, this being reflected by the fact that the Rossby number is very low, namely $Ro = \veps \ll 1$.

In the regime of fast rotations, we see from \eqref{i:eqCor} that any non-homogeneity will be amplified by the Coriolis force. We therefore expect that even slight perturbations in the fluid
density will have an impact on the dynamics. Let us explain how this happens.
The only other force that is able to compensate the effect of fast rotation is, at the geophysical scale, the pressure force, which must therefore scale as $1/\veps$:
\begin{equation*}
\frac{1}{\veps} \rho u^\perp \approx \frac{1}{\veps} \nabla \Pi.
\end{equation*}
Assume now that the fluid is quasi-homogeneous, meaning that density is nearly constant, that is $\rho = 1 + \veps R$. Then, the Coriolis force reads
\begin{equation*}
F_C = \frac{1}{\veps} u^\perp + R u^\perp.
\end{equation*}
Because the fluid is incompressible, the dominant term $\veps^{-1} u^\perp$ in this equation is in fact the gradient of some function. Therefore, it does not appear in the weak form
of the incompressible MHD equations, and the Coriolis force only ends up contributing $Ru^\perp$ to the dynamics. 

\medskip

Notice that forcing terms of the form $RU^\perp$ are in fact covered by equations \eqref{i_eq:MHD-I}, provided that the matrix $\mathfrak{C}$ is given by
\begin{equation}\label{eq:MatrixC}
\mathfrak{C} = \left(
\begin{array}{cc}
0 & -1\\
1 & 0
\end{array}
\right).
\end{equation}
Therefore, from the rough discussion above, we expect some kind of convergence theorem, from solutions of \eqref{i:MHDnHom} to those of \eqref{i_eq:MHD-I}, supplemented with $\mathfrak{C}$
given by \eqref{eq:MatrixC}. For this to work, the fact that \eqref{i_eq:MHD-I} is an ideal MHD system makes it necessary to assume that the viscosity and resistivity both vanish in the limit $\veps \rightarrow 0^+$, or in other words $h(\veps) \rightarrow 0^+$.

Our goal, in the second part of this article, will be to give rigorous grounds to the previous heuristic argument. A rigorous convergence theorem has already been shown in \cite{Cobb-F}
in the case where the viscosity and resistivity remain non-degenerate $h(\veps) \equiv 1$, with a method based on compensated compactness. In that case there is weak convergence (in a suitable topology)
\begin{equation}\label{i:Conv}
\left( \frac{1}{\veps}(1 - \rho_\veps), u_\veps, b_\veps \right) \quad \wtend \quad (R, U, B), \qquad \qquad \text{for } \veps \rightarrow 0^+,
\end{equation}
of the solutions of \eqref{i:MHDnHom} to those of the \emph{dissipative} quasi-homogeneous system
\begin{equation}\label{i:MHDvisc}
\begin{cases}
\partial_t R + \D \big( RU \big) = 0\\[1ex]
\partial_t U + \D(U \otimes U - B \otimes B) + R U^\perp + \nabla \left( \Pi + \dfrac{1}{2} |B|^2 \right) = \nu(1) \Delta U\\[1ex]
\partial_t B + \D(U \otimes B - B \otimes U) = \mu(1) \Delta B\\[1ex]
\D(U) = \div(B) = 0.
\end{cases}
\end{equation}
However, since the convergence \eqref{i:Conv} is weak, there is no explicit speed of convergence. In addition, as the method used in \cite{Cobb-F} uses bounds for the solutions of \eqref{i:MHDnHom} in the energy space $L^\infty(L^2) \cap L^2 (\dot{H}^1)$, the case of vanishing viscosity and resistivity $h(\veps) \rightarrow 0^+$ is not handled.

In this paper, we use a \emph{relative entropy} inequality (see \cite{FJN} for more on this method) for problem \eqref{i:MHDnHom} to prove quantitative convergence results of the solutions
of \eqref{i:MHDnHom} in both cases: to those of \eqref{i:MHDvisc} when $h(\veps) \equiv 1$, and those of \eqref{i_eq:MHD-I} when $h(\veps) \rightarrow 0^+$.
What this amounts to is a structure theorem: solutions of \eqref{i:MHDnHom} are the sum of a (unique regular) solution of \eqref{i_eq:MHD-I}, or \eqref{i:MHDvisc},
plus a remainder whose $L^2$ norm has limit zero as $\veps \rightarrow 0^+$.

\subsection{Overview of the paper}
Let us conclude this introduction with a short overview of the paper.

We will start, in Section \ref{WP}, by studying the well-posedness of the quasi-homogeneous system \eqref{i_eq:MHD-I} in critical Besov spaces $\B(\R^d)$,
under conditions \eqref{i_eq:Lip} and \eqref{i_eq:p}. In this context, our main results concern the local existence and uniqueness of solutions, and a continuation criterion.
As anticipated in Subsection \ref{ss:i_ideal}, a fundamental step of the proof will consist in recasting the problem in Els\"asser variables.

In Section \ref{s:res-as}, we state our main asymptotic results, concerning the rigorous derivation of the quasi-homogeneous MHD systems \eqref{i_eq:MHD-I} and \eqref{i:MHDvisc}, together with a precise
rate of convergence of the solutions.
The proof of those results will be the matter of the forthcoming sections. More precisely,
Section \ref{s:entropy} is devoted to the proof of a general relative entropy inequality for the primitive system \eqref{i:MHDnHom},
which we use in Section \ref{s:rotation} to prove our main convergence results.

An appendix containing a few tools from Littlewood-Paley analysis and transport equations will end the article.

\paragraph{Notations and conventions.}

Before starting, let us introduce some useful notation we use throughout this text.

The space domain will be denoted by $\Omega\subset\R^d$, where $d\geq 2$.
All derivatives are (weak) derivatives, and the symbols $\nabla$, $\D$ and $\Delta$ are, unless specified otherwise, relative to the space variables. When $d=2$, we set $\nabla^\perp=(-\d_2,\d_1)$
and, for a vector field $v\in\R^2$, $\curl v\,=\,\nabla\times v\,=\,\d_1v^2-\d_2v^1$.
Given a subset $U\subset\Omega$ or $U\subset\R_+\times\Omega$,
we note $\mathcal{D}(U)$ the space of compactly supported $C^\infty$ functions on $U$.
If $f$ is a tempered distribution, we note $\mathcal{F}[f] = \what{f}$ the Fourier transform of $f$ with respect to the space variables.

For $n\geq 1$, we denote by $\mc M_n(\R)$ the set of $n\times n$ matrices with entries in $\R$.

For $1 \leq p \leq +\infty$, we will note $L^p(\Omega) = L^p$ when there is no ambiguity regarding the domain of definition of the functions. Likewise, we omit the dependency on $\Omega$ in functional spaces when no mistake can be made.
If $X$ is a Fr\'echet space of functions, we note $L^p(X) = L^p(\mathbb{R}_+ ; X)$. For any finite $T > 0$, we note $L^p_T(X) = L^p([0, T] ; X)$ and $L^p_T = L^p([0, T];\R)$.
When $X$ is Banach, we will denote by $C_w([0,T];X)$ the space of functions which are continuous in time with values in $X$ endowed with its weak topology.

Let $\big(f_\veps\big)_{\varepsilon > 0}$ be a sequence of functions in a normed space $X$. If this sequence is bounded in $X$,  we use the notation $\big(f_\veps\big)_{\veps > 0} \subset X$.
If $X$ is a topological linear space, whose (topological) dual is $X'$, we note $\langle \, \cdot \, | \, \cdot \, \rangle_{X' \times X}$ the duality brackets.

Any constant will be generically noted $C$, and, whenever deemed useful, we will specify the dependencies by writing $C = C(a_1, a_2, a_3, ...)$. We agree to write $A \lesssim B$ whenever there is an irrelevant constant $C > 0$ such that $A \leq C B$. 


\subsection*{Acknowledgements}

{\small
The work of the second author has been partially supported by the LABEX MILYON (ANR-10-LABX-0070) of Universit\'e de Lyon, within the program ``Investissement d'Avenir''
(ANR-11-IDEX-0007),  and by the projects BORDS (ANR-16-CE40-0027-01) and SingFlows (ANR-18-CE40-0027), all operated by the French National Research Agency (ANR).
}


\section{Well-posedness of the quasi-homogeneous ideal MHD system} \label{WP}

In this section, we state and prove our main results concerning the well-posedness theory of problem \eqref{i_eq:MHD-I} in critical Besov spaces $B^s_{p,r}$, under conditions \eqref{i_eq:Lip}
and \eqref{i_eq:p} on the indices.

\subsection{Statement of the well-posedness results} \label{ss:WP-res}

Here we state our main well-posedness results concerning problem \eqref{i_eq:MHD-I}. The first theorem is about
local in time existence and uniqueness of solutions for initial data in all Besov spaces contained in the space of Lipschitz functions,
so that the case of the critical spaces $B^{d/p + 1}_{p, 1}$ is covered by this theorem. Recall that $1<p<+\infty$ throughout this paper.

\begin{thm}\label{th:BesovWP}
Let $(s,p,r)\in\R\times\,]1,+\infty[\,\times[1,+\infty]$ such that $s \geq 1 + d/p$, with $r = 1$ if $s = 1 + d/p$. Let $(R_0,U_0, B_0,)$ be a set of initial data which lie in $\B(\R^d)$.

Then, there exists a time $T > 0$ such that problem \eqref{i_eq:MHD-I} has a unique solution $(R, U, B)$ in the class $C^0\big([0,T];\B(\R^d;\R^{1+2d})\big)$ if $r<+\infty$, in the class
$C^0_w\big([0,T];B^s_{p,\infty}(\R^d;\R^{1+2d})\big)$ if $r=+\infty$. Moreover, $\d_tR$, $\d_tU$ and $\d_tB$ belong to the space $C^0\big([0,T];B^{s-1}_{p,r}(\R^d;\R^d)\big)$, while
$\nabla\Pi$ belongs to $C^0\big([0,T];\B(\R^d)\big)$ (with the usual modification when $r=+\infty$).
Finally, if $T^*>0$ denotes the lifespan of this solution, then one has
\begin{equation*}
T^* \geq \frac{C}{\big\| (R_0, U_0, B_0) \big\|_{\B}}\,,
\end{equation*}
where $C > 0$ depends only on the dimension $d$ and on $(s, p, r)$.
\end{thm}


We now state a continuation/blow-up criterion for the solutions found in the previous theorem, in the spirit of the classical Beale-Kato-Majda criterion
for solutions to the $3$-D incompressible Euler equations \cite{BKM}. As one can expect from the structure of the system, the criterion
rests on the control of the norm $L^1_T(\rm Lip)$ of the vector fields $U$ and $B$.

\begin{thm} \label{th:cont-crit}
Let $(s, p, r)$ be as in Theorem \ref{th:BesovWP} and let $(R, U, B) \in C^0\big([0,T[\,;\B(\R^d;\R^{1+2d})\big)$, with $C^0$ replaced by $C^0_w$ if $r = +\infty$, be a solution of \eqref{i_eq:MHD-I}. Assume that 
\begin{equation*}
\int_0^T \Big( \big\| \nabla U(t) \big\|_{L^\infty} + \big\| \nabla B(t) \big\|_{L^\infty} \Big) \dt < +\infty.
\end{equation*}
Then $(R, U, B)$ can be continued beyond $T$ into a solution of \eqref{i_eq:MHD-I} with the same regularity.
\end{thm}

The rest of this section is devoted to the proof of the previous statements.
We will proceed in several steps. The first step (see Subsection \ref{ss:Elsasser}) is to pass to the Els\"asser formulation of the equations and to establish the equivalence of this formulation with the
original one in our functional framework.
Then, the results will be recasted in Els\"asser variables. In Subsection \ref{ss:abAprioriPFinite}, we will mainly focus on the proof of \tsl{a priori} estimates
for regular solutions to equations \eqref{i_eq:MHD-I}, and sketch only the other steps of the proof of the local well-posedness result. Finally, in Subsection \ref{ss:ContCritPFinite}
we present the proof of the continuation criterion, \tsl{i.e.} Theorem \ref{th:cont-crit}.


\subsection{Passing to the Els\"asser formulation} \label{ss:Elsasser}

Let us introduce the so-called \emph{Els\"asser variables}, defined by the transformation
\[ 
\alpha = U + B \qquad \text{ and } \qquad \beta = U - B.
\] 
In the new set of unknowns $(R,\alpha,\beta)$, the quasi-homogeneous ideal MHD system \eqref{i_eq:MHD-I} can be recasted in the following form:
\begin{equation}\label{eq:MHDab}
\begin{cases}
\partial_t R + \D \left( \dfrac{1}{2} R(\al + \bt) \right) = 0 \\[1ex]
\partial_t \alpha + (\beta \cdot \nabla) \alpha + \dfrac{1}{2}  R \mathfrak{C} (\al + \bt) + \nabla \pi_1 = 0\\[1ex]
\partial_t \beta + (\alpha \cdot \nabla) \beta + \dfrac{1}{2}  R \mathfrak{C} (\al + \bt) + \nabla \pi_2  = 0\\[1ex]
\D(\alpha) = \D(\beta) = 0\,,
\end{cases}
\end{equation}
where $\pi_1$ and $\pi_2$ are (possibly distinct) scalar ``pressure'' functions. In fact, as explained in the sequel, in our framework we must have $\nabla \pi_1 = \nabla \pi_2$.

More precisely, it has been proved in \cite{Cobb-F_Els} that, under some (very mild) integrability assumptions on the set of unknowns $(R,U,B)$ and $(R,\alpha,\beta)$, the two formulations \eqref{i_eq:MHD-I}
and \eqref{eq:MHDab} are completely equivalent. It is worth mentioning that the result of \cite{Cobb-F_Els} requires very weak regularity on the solutions (it is in fact stated
for weak solutions of the two systems), so it works for a very large class of data. In our framework, the result of \cite{Cobb-F_Els} can be rephrased as follows.

\begin{thm}\label{th:symm}
Let $(s,p,r)\in\R\times[1,+\infty]\times[1,+\infty]$ satisfy conditions \eqref{i_eq:Lip} and \eqref{i_eq:p}. Let $T>0$.
\begin{enumerate}[(i)]
 \item Consider $(R_0,U_0,B_0)\in \left(B^{s}_{p,r}\right)^3$ a set of initial data, with $U_0$ and $B_0$ being divergence-free, and let $(R,U,B)\in \left(L^\infty_T(B^s_{p,r})\right)^3$ be a strong
solution to \eqref{i_eq:MHD-I} related to that initial datum, for some pressure $\Pi$. Define $(\al_0, \bt_0) = (U_0+B_0, U_0-B_0)$ and $(\alpha,\beta)=(U+B,U-B)$. \\
Then $(R,\alpha,\beta)\in\left(L^\infty_T(B^s_{p,r})\right)^3$ is a strong solution to system \eqref{eq:MHDab} related to the initial datum $(R_0,\alpha_0,\beta_0)$,
for suitable ``pressure'' functions $\pi_1$ and $\pi_2$ such that $\nabla\pi_1=\nabla\pi_2=\nabla\pi$, where $\pi$ is defined in \eqref{eq:p-P}.
\item {Conversely}, consider $(R_0,\alpha_0,\beta_0)\in \left(B^{s}_{p,r}\right)^3$, with both $\alpha_0$ and $\beta_0$ being divergence-free, and let
$(R,\alpha,\beta)\in \left(L^\infty_T(B^s_{p,r})\right)^3$ be a related strong solution to \eqref{eq:MHDab}, for suitable ``pressures'' $\pi_1$ and $\pi_2$.
Define $(U_0,B_0)=\big(\frac{\alpha_0+\beta_0}{2},\frac{\alpha_0-\beta_0}{2}\big)$ and $(U,B)=\big(\frac{\alpha+\beta}{2},\frac{\alpha-\beta}{2}\big)$. \\
Then one has $\nabla\pi_1=\nabla\pi_2$, and $(R,U,B)\in \left(L^\infty_T(B^s_{p,r})\right)^3$ is a strong solution to system \eqref{i_eq:MHD-I} related to the initial datum $(R_0,U_0,B_0)$,
for a suitable pressure function $\Pi$ such that $\nabla\Pi\,=\,\nabla\big(\pi_1-|B|^2/2\big)$.
\end{enumerate}

In addition, in that case the Els\"asser variables $(R, \al, \bt)$ solve the following system
\begin{equation}\label{eq:EquivEq}
\begin{cases}
\partial_t R + \dfrac{1}{2}\D \big( R (\al + \bt) \big) = 0 \\[1ex]
\partial_t \al + \mathbb{P} \D (\bt \otimes \al) + \dfrac{1}{2} \mathbb{P} \big( R \mathfrak{C} (\al + \bt) \big) = 0\\[1ex]
\partial_t \bt + \mathbb{P} \D (\al \otimes \bt) + \dfrac{1}{2} \mathbb{P} \big( R \mathfrak{C} (\al + \bt) \big) = 0
\end{cases}
\end{equation}
almost everywhere in $[0,T]\times\R^d$, where $\P$ denotes the Leray projector onto the space of divergence-free vector fields.
\end{thm}

Recall that the Leray projector $\mathbb{P}$, which is the $L^2$-orthogonal projection on the subspace of divergence-free functions, is defined as a Fourier multiplier by
\begin{equation*}
\forall f \in L^2(\mathbb{R}^d ; \mathbb{R}^d), \qquad \what{(\mathbb{P}f)_j}(\xi) = \sum_{k=1}^d\left( 1 - \frac{\xi_j \xi_k}{|\xi|^2} \right) \what{f_k}(\xi)\,.
\end{equation*}
In other words, $\mathbb{P} = I + \nabla (-\Delta)^{-1}\D$, in the sense of pseudo-differential operators.
We refer to Remark \ref{r:Leray} in the Appendix for continuity properties of the Leray projector in Lebesgue and Besov spaces.


Finally, by introducing commutators of operators, we see that \eqref{eq:EquivEq} is in fact a system of transport equations: 
more precisely, we can recast it in the form
\begin{equation}\label{eq:MHDabAPriori}
\begin{cases}
\partial_t R + \dfrac{1}{2}\,\D \big( R(\al + \bt) \big) = 0 \\[1ex]
\partial_t \alpha + (\beta \cdot \nabla) \alpha + \dfrac{1}{2} \mathbb{P} \big( R \mathfrak{C} (\al + \bt) \big) = \big[ \beta \cdot \nabla, \, \mathbb{P} \big]\alpha \\[1ex]
\partial_t \beta + (\alpha \cdot \nabla) \beta + \dfrac{1}{2} \mathbb{P} \big( R \mathfrak{C} (\al + \bt) \big) = \big[ \alpha \cdot \nabla, \, \mathbb{P} \big]\beta\,,
\end{cases}
\end{equation}
which we equip with a regular initial datum $(R_0, \al_0, \bt_0)$, with $\D(\al_0)=\D(\bt_0)=0$. 

Here the transport operators are to be understood in the weak sense: thanks to the divergence-free conditions,
\begin{equation*}
(\bt \cdot \nabla) \al = \D (\bt \otimes \al) \qquad\qquad \text{ and } \qquad\qquad \big[ \bt \cdot \nabla, \mathbb{P} \big] \al = (I - \mathbb{P}) \D (\bt \otimes \al).
\end{equation*}
However, since we will solve the system only in the case where $\al$ and $\bt$ are Lipschitz functions, we can work with either definition (weak or strong) of the transport operators.

System \eqref{eq:MHDabAPriori} is the final (equivalent, in our framework) formulation of the original system \eqref{i_eq:MHD-I}, which we are going to work with in our proof.

\subsection{\tsl{A priori} estimates for the Els\"asser variables}\label{ss:abAprioriPFinite}

In this subsection, we show \tsl{a priori} estimates in Besov spaces $B^s_{p,r}$, under conditions 
\eqref{i_eq:Lip} and \eqref{i_eq:p} on the indices, for smooth solutions to system \eqref{eq:MHDabAPriori}.

\begin{prop} \label{p:a-priori}
Let $(s,p,r)\in\R\times[1,+\infty]\times[1,+\infty]$ satisfy conditions \eqref{i_eq:Lip} and \eqref{i_eq:p}.
Let $(R,\al, \bt)$ be regular solutions of system \eqref{eq:MHDabAPriori}, related to smooth initial data
$(R_0,\al_0, \bt_0)$, with $\al_0$ and $\bt_0$ being divergence-free.

Then there exist two constants $C_1, C_2 > 0$, which depend on the dimension and $(s, p, r)$, as well as a time $T^*>0$, which depends on the quantities above and on $\big\|\big(R_0,\al_0, \bt_0\big)\big\|_{\B}$, such that
\begin{equation} \label{est:B-bound}
\sup_{t\in[0,T^*]}\big\|\big(R(t),\al(t), \bt(t)\big)\big\|_{\B}\, \leq\, C_1\, \big\|\big(R_0,\al_0,\bt_0\big)\big\|_{\B}\,,
\end{equation}
together with the lower bound
\begin{equation*}
T^*\, \geq\, \frac{C_2}{\big\|\big(R_0,\al_0, \bt_0\big)\big\|_{\B}}\,.
\end{equation*}
\end{prop}

Before proving the previous proposition, we recall that, since $p\in\,]1,+\infty[\,$, the Leray projector $\P$ is well defined and acts continuously on
$L^p$ and on $B^s_{p,r}$, for any $s\in \mathbb{R}$ and $r \in [1, +\infty]$. See also Remark \ref{r:Leray} in the Appendix about this point.

We also recall that, under conditions \eqref{i_eq:Lip} and \eqref{i_eq:p}, the space $\B$ is contained in the space $W^{1, \infty}$ of globally Lipschitz functions, and that both $\B$ and $B^{s-1}_{p, r}$
are Banach algebras (see Corollary \ref{c:tame} in the Appendix).

The first step in the proof of the previous statement is to find a $\B$-bound for the commutators in the right-hand side of \eqref{eq:MHDabAPriori}: this is the scope of the next lemma.

\begin{lemma}\label{l:commB}
Let $(s,p,r)\in\R\times[1,+\infty]\times[1,+\infty]$ satisfy conditions \eqref{i_eq:Lip} and \eqref{i_eq:p}. For all divergence-free vector fields $\al, \bt \in \B$, we have
\begin{equation*}
\left\|  \big[ \bt \cdot \nabla , \mathbb{P} \big] \al  \right\|_{\B}\, \lesssim\,
\| \nabla \al \|_{L^\infty} \| \bt \|_{\B} + \|\al\|_{\B} \| \nabla \bt \|_{L^\infty}\,.
\end{equation*}
\end{lemma}

\begin{proof}[Proof of Lemma \ref{l:commB}]
The estimate is based on the Bony decomposition (see Section \ref{s:NHPC} in the Appendix) for the two products appearing in the commutator, which yields
\begin{equation*}
\big[ \bt \cdot \nabla , \mathbb{P} \big] \al = \big[ \mathcal{T}_{\beta_k} \partial_k , \mathbb{P} \big]  \al + \mathcal{T}_{\partial_k \mathbb{P} \al} (\beta_k) - \mathbb{P} \mathcal{T}_{\partial_k \al} (\beta_k) + \mathcal{R} \big( \beta_k, \partial_k \mathbb{P} \al \big) -  \mathbb{P} \mathcal{R} \big( \beta_k, \partial_k \al \big),
\end{equation*}
where an implicit summation is made over the repeated index $k = 1, ..., d$. To begin with, Lemma \ref{l:ParaComm} allows us to handle the first summand:
\begin{equation*}
\left\| \big[ \mathcal{T}_{\beta_k}  , \mathbb{P} \big] \partial_k \al  \right\|_{\B} \lesssim \| \nabla \al \|_{B^{s-1}_{p, r}}   \| \nabla \bt \|_{L^\infty} \lesssim \| \al \|_{\B}  \| \nabla \bt \|_{L^\infty} .
\end{equation*}
Next, by using Proposition \ref{p:op}, 
we get
\begin{align*}
&\| \mathcal{T}_{\partial_k \mathbb{P} \al}(\bt_k)\|_{\B} = \| \mathcal{T}_{\partial_k \al}(\bt_k)\|_{\B} \lesssim \|\nabla \al\|_{L^\infty} \|\bt\|_{\B} \\ 
&\big\| \mathbb{P} \mathcal{T}_{\partial_k \al} (\beta_k)  \big\|_{\B} \lesssim \| \nabla \al \|_{L^\infty} \| \beta \|_{\B}\,,
\end{align*}
where the first inequality holds because $\mathbb{P}\al = \al$ (the vector field $\al$ is divergence free by assumption) and the second one holds because $\P$ is bounded on $\B$ (since $1<p<+\infty$).
Owing to those same properties, another application of Proposition \ref{p:op} gives
\begin{align*}
\big\| \mathcal{R} ( \beta_k, \partial_k \mathbb{P} \al ) \big\|_{B^{s}_{p, r}} &\lesssim \| \nabla \al \|_{B^0_{\infty, \infty}}  \| \bt \|_{\B} \lesssim \| \nabla \al \|_{L^\infty} \| \bt \|_{\B} \\
\big\| \mathbb{P} \mathcal{R} (\bt_k, \partial_k \al) \big\|_{\B} &\lesssim \| \nabla \al \|_{L^\infty} \| \bt \|_{\B}\,,
\end{align*}
where we also used the embedding $L^\infty  \hookrightarrow B^0_{\infty, \infty}$.
Note that we always have $s-1>0$ under our assumptions, because $1 < p < +\infty$.

The lemma is thus proved.
\end{proof}

We are finally ready to prove Proposition \ref{p:a-priori}. 

\begin{proof}[Proof of Proposition \ref{p:a-priori}]
We fix an index $j \geq -1$ and apply the dyadic block $\Delta_j$ to the three equations of system \eqref{eq:MHDabAPriori}. Recalling that $\al$ and $\bt$ are divergence-free, we get
\begin{equation*}
\begin{cases}
\partial_t \Delta_j R + \frac{1}{2} (\al + \bt) \cdot \nabla \Delta_j R = \frac{1}{2} \big[ (\al + \bt) \cdot \nabla, \Delta_j \big] R \\
\partial_t \Delta_j \al + (\bt \cdot \nabla ) \Delta_j \al = \big[ \bt \cdot \nabla, \, \Delta_j\big] \al + \Delta_j \big[ \beta \cdot \nabla, \, \mathbb{P} \big]\al - \frac{1}{2} \Delta_j \mathbb{P} \big( R \mathfrak{C} (\al + \bt) \big) \\
\partial_t \Delta_j \bt + (\al \cdot \nabla ) \Delta_j \bt = \big[ \al \cdot \nabla, \, \Delta_j\big] \bt + \Delta_j \big[ \al \cdot \nabla, \, \mathbb{P} \big]\bt - \frac{1}{2} \Delta_j \mathbb{P} \big( R \mathfrak{C} (\al + \bt) \big)\,.
\end{cases}
\end{equation*}
Now, Lemma \ref{l:CommBCD} gives the following bounds: there exists a sequence $\big(c_j(t)\big)_{j\geq-1}$ in the unit ball of $\ell^r$ such that
\begin{equation*}
\begin{split}
&
2^{sj} \Big( \left\| \big[ \bt \cdot \nabla, \, \Delta_j\big] \al \right\|_{L^p} + \left\| \big[ \al \cdot \nabla, \, \Delta_j\big] \bt \right\|_{L^p} + \left\| \big[ (\al + \bt) \cdot \nabla, \Delta_j \big] R \right\|_{L^p} \Big)\\
& \qquad\qquad\qquad\qquad\qquad\qquad\qquad\qquad
\lesssim c_j(t) \left( \|\al\|_{\B} \|\bt\|_{\B} + \|R\|_{\B} \|\al + \bt \|_{\B} \right)\,.
 \end{split}
\end{equation*}
On the other hand, Lemma \ref{l:commB} and the embedding $\B \hookrightarrow W^{1, \infty}$ give a similar inequality: there exists a (different) sequence
$\big(c_j(t)\big)_j$ in the unit ball of $\ell^r$ such that
\begin{equation*}
\left\| \Delta_j \big[ \bt \cdot \nabla , \mathbb{P} \big] \al  \right\|_{L^p} + \left\| \Delta_j \big[ \al \cdot \nabla , \mathbb{P} \big] \bt  \right\|_{L^p} \lesssim c_j(t) \|\al\|_{\B} \| \bt \|_{\B}\,.
\end{equation*}
Finally, all that remains is the rotation term $\frac{1}{2}R \mathfrak{C}(\al + \bt)$: using the fact that $\B$ is a Banach algebra and that $\mathbb{P}$ is bounded on $\B$, we get
\begin{equation*}
\left\| \mathbb{P} \big( R \mathfrak{C}(\al + \bt) \big) \right\|_{\B} \lesssim \|R\|_{\B} \|\al + \bt \|_{\B}\,.
\end{equation*}

Therefore, basic $L^p$ estimates for transport equations with divergence-free vector fields give
\begin{align}
&2^{js} \big\| \Delta_j \big(R(t),\al(t), \bt(t)\big) \big\|_{L^p} \label{eq:ContInLp} \\
&\qquad\qquad\qquad\qquad
\lesssim 2^{js} \big\|\Delta_j \big(R_0,\al_0, \bt_0, R_0\big)\big\|_{L^p} +
\int_0^t c_j(\tau) \big\| \big(R(\tau),\al(\tau), \bt(\tau)\big) \big\|_{\B}^2\, {\rm d}\tau\,. \nonumber
\end{align}
Taking the $\ell^r$ norm, we find, for all times $t> 0$, the inequality
\begin{equation*}
\big\|  \big(R(t),\al(t), \bt(t)\big) \big\|_{L^\infty_t(\B)} \lesssim \|(\al_0, \bt_0, R_0)\|_{\B} + \left\| \int_0^t c_j(\tau) \big\| \big(R(\tau)\,\al(\tau), \bt(\tau)\big) \big\|_{\B}^2 {\rm d} \tau \right\|_{\ell^r}\,.
\end{equation*}
Using the Minkowski inequality (see Proposition 1.3 in \cite{BCD}) to slip the $\ell^r$ norm inside the integral, we finally deduce the following bound:
\begin{equation} \label{est:bound-Besov}
c\, \big\|  \big(R(t),\al(t), \bt(t)\big) \big\|_{\B} \leq \|(R_0,\al_0, \bt_0)\|_{\B} + \int_0^t \big\| \big(R(\t),\al(\tau), \bt(\tau)\big) \big\|_{\B}^2 {\rm d} \tau\,.
\end{equation}
for some constant $c > 0$ depending only on $(d, s, p, r)$.

From this inequality, it is easy to obtain an estimate like \eqref{est:B-bound}, in some interval $[0,T^*]$. For this, set $E(t) = \big\|  \big(R(t),\al(t), \bt(t)\big) \big\|_{\B}$ and define the time $T^*$ by 
\begin{equation*}
T^* = \sup \left\{ T > 0\; \Big|\quad  \int_0^T E(t)^2 \dt < E(0) \right\}\,.
\end{equation*}
Then, from \eqref{est:bound-Besov} we get estimate \eqref{est:B-bound} in $[0,T^*]$:
\begin{equation*}
\forall\,t\in[0,T^*]\,,\qquad\qquad 
E(t)\, \leq\, \frac{2}{c}\,E(0)\,.
\end{equation*}
In turn, this latter inequality provides also a lower bound for $T^*$. Indeed, for all $t \leq T^*$ we have
\begin{equation*}
\int_0^t E(\tau) {\rm d} \tau\, \leq\, \left( \frac{2}{c}\,E(0)\right)^2\,t\,,
\end{equation*}
so that, by definition of $T^*$, we must have $T^* \geq c^2 / \big(4\,E(0)\big)$.

This ends the proof of the proposition.
\end{proof}

We conclude this subsection by studying the regularity of the pressure term in system \eqref{eq:MHDab}. The next result will be fundamental in order to establish
regularity of the hydrodynamic pressure, namely $\nabla\Pi$ in system \eqref{i_eq:MHD-I}.

\begin{prop}\label{p:pressure}
Consider $(d, s, p, r)$, $(R_0, \al_0, \bt_0)$ and $(R,\al,\bt)$ exactly as in Proposition \ref{p:a-priori}. Let $T^*>0$ be defined as in that same proposition.

Then, the pressure terms $\pi_1$ and $\pi_2$ are such that $\nabla\pi_1\,=\,\nabla\pi_2\,\in L^\infty_{T^*}(\B)$, with the estimate
\begin{equation}
\| \nabla \pi_1 \|_{L^\infty_{T^*}(\B)}\, \lesssim\, \big\| (R_0, \al_0, \bt_0) \big\|_{\B}^2\,.
\end{equation}
\end{prop}

\begin{proof}
We already know from Theorem \ref{th:symm} that, in our functional framework, one has $\nabla\pi_1=\nabla\pi_2$. So, it is enough to establish Besov bounds for that quantity.

System \eqref{eq:MHDab} may be rewritten as
\begin{equation*}
\partial_t \al + \D (\al \otimes \bt) + \frac{1}{2} R \mathfrak{C} (\al + \bt) = - \nabla \pi_1\,.
\end{equation*}
By applying the projector $\Q\,:=\,{\rm Id} - \mathbb{P}$ to this equation, we immediately have a $B^{s-1}_{p, r}$ estimate for $\nabla \pi_1$: since $B^{s-1}_{p, r}$ is a Banach algebra (by condition \eqref{i_eq:Lip} above and Corollary \ref{c:tame} below) on which the operator $\Q$ acts continuously, it follows that
\begin{equation}\label{eq:PressureLFestimate}
\| \nabla \pi_1 \|_{L^\infty_T(B^{s-1}_{p, r})} \lesssim \big\| R \mathfrak{C} (\al + \bt) \big\|_{L^\infty_T(\B)} + \| \al \otimes \bt \|_{L^\infty_T(\B)} \lesssim \big\| (R, \al, \bt) \big\|_{L^\infty_T(\B)}^2\,.
\end{equation}

The core of this proof is therefore finding $\B$ regularity. This can be done with the help of algebraic cancellations induced by the fact that $\al$ and $\bt$ are divergence free.
By taking the divergence of equation \eqref{eq:MHDab}, we see that the Laplacian $\Delta \pi_1$ reads
\begin{equation*}
- \Delta \pi_1 = \frac{1}{2} \D \big( R \mathfrak{C} (\al + \bt) \big) + \partial_i \partial_j\left(\al_i \bt_j\right)\,,
\end{equation*}
where there is an implicit sum on both of the repeated indices $i, j \in \{ 1, ..., d \}$. We will estimate $\Delta \pi_1$ in the space $B^{s-1}_{p, r}$. Owing to the fact that both $\al$ and $\bt$ are divergence free, we have
\begin{equation*}
\partial_i \partial_j (\al_i \bt_j) = \partial_i \al_j \partial_j \bt_i\,,
\end{equation*}
so that, thanks to the tame estimates from Corollary \ref{c:tame}, we have
\begin{equation}\label{eq:PressureHFestimate}
\| \Delta \pi_1 \|_{B^{s-1}_{p, r}} \lesssim \| R \|_{\B} \| \al + \bt \|_{\B} + \| \nabla \al \|_{B^{s-1}_{p, r}}  \| \nabla \bt \|_{B^{s-1}_{p, r}}\,.
\end{equation}

To end the proof, it only remains to combine the low order estimate \eqref{eq:PressureLFestimate} with the high order one \eqref{eq:PressureHFestimate}. We have, by separating low and high frequencies in $\nabla \pi_1$,
\begin{equation*}
\| \nabla \pi_1 \|_{\B} \lesssim \| \Delta_{-1} \nabla \pi_1 \|_{\B} + \big\| ({\rm Id} - \Delta_{-1}) \nabla \pi_1 \big\|_{\B}\,,
\end{equation*}
and the Bernstein inequalities then give
\begin{equation*}
\| \nabla \pi_1 \|_{\B} \lesssim \| \nabla \pi_1 \|_{B^{s-1}_{p, r}} + \| \Delta \pi_1 \|_{B^{s-1}_{p, r}}\,,
\end{equation*}
which finally implies the result.
\end{proof}

\subsection{Proof of the continuation criterion}\label{ss:ContCritPFinite}

In this section, we seek to prove a continuation criterion for the solutions of \eqref{eq:MHDab}, or equivalently of \eqref{eq:MHDabAPriori}, in $\B$.
Recall that we always assume conditions \eqref{i_eq:Lip} and \eqref{i_eq:p} on the indices $(s,p,r)$.

\begin{prop}\label{p:ContFinite}
Let $(R_0,\al_0, \bt_0) \in \left(\B\right)^3$, with $\D(\al_0) = \D(\bt_0) = 0$. Given a time $T > 0$, let $(R,\al, \bt)$ be a solution of \eqref{eq:MHDabAPriori} on $[0,T[\,$, related
to that initial datum and belonging to $\left(L^\infty_t(\B)\right)^3$ for any $0\leq t<T$. Assume that
\begin{equation}\label{eq:ContCritPFinite}
\int_0^T \Big( \|\nabla \al \|_{L^\infty} + \|\nabla \bt \|_{L^\infty} \Big) \dt < +\infty\,.
\end{equation}
Then $(R,\al, \bt)$ can be continued beyond $T$ into a solution of \eqref{eq:MHDabAPriori} with the same regularity.
\end{prop}

\begin{rmk}
This theorem implies that the lifespan of solutions in $\B$ (with $(s, p, r)$ satisfying the Lipschitz condition, \tsl{i.e.} condition \eqref{i_eq:Lip} above) does not depend on the values of $(s, p, r)$,
thanks to the embeddings $\B \hookrightarrow B^{1 - d/p}_{\infty, r} \hookrightarrow B^1_{\infty, 1} \hookrightarrow W^{1, \infty}$.
\end{rmk}

\begin{proof}
By standard arguments, it is enough to show that a solution $(R,\al, \bt)$ remains bounded in $L^\infty_T(\B)$ as long as $T$ satisfies \eqref{eq:ContCritPFinite}.
This is mainly a matter of rewriting the inequalities of the previous proof in a more precise way.

Let $(R,\al, \bt)$ be a solution of \eqref{eq:MHDabAPriori} in $[0,T[\,$, related to an initial datum as in Proposition \ref{p:ContFinite}. We already know from Lemma \ref{l:CommBCD} that
\begin{multline}\label{eq:ContCommut2}
2^{js} \Big( \left\| \big[ \bt \cdot \nabla, \, \Delta_j\big] \al \right\|_{L^p} + \left\| \big[ \al \cdot \nabla, \, \Delta_j\big] \bt \right\|_{L^p} + \left\| \big[ (\al + \bt) \cdot \nabla, \Delta_j \big] R \right\|_{L^p} \Big)\\
\lesssim c_j(t) \left( \|\nabla \al\|_{L^\infty} \|\bt\|_{\B} + \|\al\|_{\B} \| \nabla \bt \|_{L^\infty}\right. \\
\qquad \qquad + \left.\|R\|_{\B} \|\nabla (\al + \bt)\|_{L^\infty} + \|\al + \bt\|_{\B} \|\nabla R\|_{L^\infty} \right)\,,
\end{multline}
for some sequence $\big(c_j(t)\big)_j$ belonging to the unit sphere of $\ell^r$. Similarly, from Lemma \ref{l:commB} we have
\begin{multline}\label{eq:ContCommut}
2^{js} \Big( \left\| \Delta_j \big[ \bt \cdot \nabla , \mathbb{P} \big] \al  \right\|_{L^p} + \left\| \Delta_j \big[ \al \cdot \nabla , \mathbb{P} \big] \bt  \right\|_{L^p} \Big) \\
\lesssim c_j(t) \left( \| \nabla \al \|_{L^\infty} \| \bt \|_{\B} + \|\al\|_{\B} \| \nabla \bt \|_{L^\infty}\right).
\end{multline}

Finally, we look at the rotation term $R \mathfrak{C}(\al + \bt)$. We use here the paraproduct decomposition \eqref{eq:bony}, but we need to do better than the tame estimates
of Corollary \ref{c:tame} if we want to avoid the appearance of an unpleasant $L^\infty$ norm on $\alpha$ and $\beta$.
We start by using Proposition \ref{p:op} to obtain
\begin{equation*}
\left\|\mc T_{R}\mf C(\al+\bt)\right\|_{\B}\,+\,\left\|\mc R\big(R,\mf C(\al+\bt)\big)\right\|_{\B}\,\lesssim\,\left\|R\right\|_{L^\infty}\,\left\|\al+\bt\right\|_{\B}\,,
\end{equation*}
where we have used also the embedding $L^\infty\hookrightarrow B^0_{\infty,\infty}$. Next, we write the remaining term as
$$
\mc T_{\mf C(\al+\bt)}R\,=\,\mc T_{\Delta_{-1}\mf C(\al+\bt)}R\,+\,\mc T_{(\Id-\Delta_{-1})\mf C(\al+\bt)}R\,.
$$
By Proposition \ref{p:op} again and Bernstein inequalities, we easily get
\[
\left\|\mc T_{\Delta_{-1}\mf C(\al+\bt)}R\right\|_{\B}\,\lesssim\,\left\|\Delta_{-1}(\al+\bt)\right\|_{L^\infty}\,\|R\|_{\B}\,\lesssim\,\left\|(\al+\bt)\right\|_{L^p}\,\|R\|_{\B}\,.
\]
Next, we write
$$
\mc T_{(\Id-\Delta_{-1})\mf C(\al+\bt)}R\,=\,\sum_{j}S_{j-1}(\Id-\Delta_{-1})\mf C(\al+\bt)\;\Delta_jR\,.
$$
Lemma 2.84 of \cite{BCD} states that it is enough to bound the $L^p$ norm of each term of the previous sum. At this point, we observe that, for each $j\in\N$, the (smooth) symbol of the operator
$S_j(\Id-\Delta_{-1})$ is supported away from the origin. Hence, we can apply the second Bernstein inequality to get
\begin{align*}
\left\|S_{j-1}(\Id-\Delta_{-1})\mf C(\al+\bt)\;\Delta_jR\right\|_{L^p}\,&\lesssim\,\left\|S_{j-1}(\Id-\Delta_{-1})\nabla\mf C(\al+\bt)\right\|_{L^\infty}\,\left\|\Delta_j R\right\|_{L^p} \\
&\lesssim\,\left\|\nabla(\al+\bt)\right\|_{L^\infty}\,\left\|\Delta_j R\right\|_{L^p}\,.
\end{align*}
From those bounds, we infer that
$$
\left\|\mc T_{\mf C(\al+\bt)}R\right\|_{\B}\,\lesssim\,\big(\left\|\al+\bt\right\|_{L^p}\,+\,
\left\|\nabla(\al+\bt)\right\|_{L^\infty}\big)\,\left\|R\right\|_{\B}\,,
$$
which finally implies
\begin{equation} \label{est:rotation}
\left\|R\, \mathfrak{C}(\al + \bt)\right\|_{\B}\,\lesssim\,
\left\|R\right\|_{L^\infty}\,\left\|\al+\bt\right\|_{\B}\,+\,
\big(\left\|\al+\bt\right\|_{L^p}\,+\,
\left\|\nabla(\al+\bt)\right\|_{L^\infty}\big)\,\left\|R\right\|_{\B}
\end{equation}

Putting estimates \eqref{eq:ContCommut2}, \eqref{eq:ContCommut} and \eqref{est:rotation} together, we are led to the following bound:
\begin{multline}\label{eq:ContGronPFinite2}
\big\|\big(R(t),\al(t), \bt(t)\big)\big\|_{\B} \\ \lesssim \big\| (R_0,\al_0, \bt_0) \big\|_{\B}
+\int^t_0\big(\|\nabla\al\|_{L^\infty}+\|\nabla\bt\|_{L^\infty}\big)\,\big\|\big(R,\al,\bt\big)\big\|_{\B} {\rm d}\t + \\
 + \int_0^t \left\{\left(\|R\|_{L^\infty}+\|\nabla R\|_{L^\infty} \right) \|\al + \bt\|_{\B} + \|R\|_{\B} \|\al + \bt\|_{L^p}\right\} {\rm d} \tau\,.
\end{multline}

We now seek for bounds for the $L^p$ and $L^\infty$ norms appearing in the last line of \eqref{eq:ContGronPFinite2}.
First of all, because $R$ solves a pure transport equation by the divergence-free vector field $\frac{1}{2}(\al + \bt)$, we have the bound
\begin{equation} \label{est:R_inf}
\|R(t)\|_{L^\infty} \leq \|R_0\|_{L^\infty} \leq \|R_0\|_{\B}\,,
\end{equation}
thanks to the embedding $\B \hookrightarrow B^{s-1}_{p, r} \hookrightarrow L^\infty$. Next, differentiating the equation for $R$, \tsl{i.e.} the first equation
of \eqref{eq:MHDabAPriori}, with respect to $x_j$, for all $j=1\ldots d$, we find
$$
\d_t\d_jR\,+\,\frac{1}{2}\,(\al+\bt)\cdot\nabla\d_jR\,=\,-\d_j(\al+\bt)\cdot\nabla R\,,
$$
from which we immediately deduce the estimate
\begin{equation} \label{est:DR_inf}
\left\|\nabla R(t)\right\|_{L^\infty}\,\leq\,\left\|R_0\right\|_{\B}\,\exp\left(\int^t_0\left(\|\nabla\al\|_{L^\infty}+\|\nabla\bt\|_{L^\infty}\right)\,{\rm d}\t\right)\,.
\end{equation}
Finally, after remarking that
$$
\left\|[\bt \cdot \nabla, \mathbb{P}]\al\right\|_{L^p}\,\leq\,\left\|\bt\cdot\nabla\al\right\|_{L^p}\,+\,\left\|\P\left(\bt\cdot\nabla\al\right)\right\|_{L^p}\,\leq\,C\,\left\|\bt\right\|_{L^p}\,\left\|\nabla \al\right\|_{L^\infty}
$$
and analogously for the other commutator term $[\al \cdot \nabla, \mathbb{P}]\bt$, standard $L^p$ estimates for transport equation with divergence-free vector fields yield
\begin{equation} \label{est:a-b_p}
\left\|\big(\al(t),\bt(t)\big)\right\|_{L^p}\,\leq\,\left\|\big(\al_0,\bt_0\big)\right\|_{L^p}\,+\,C\int^t_0\big(\|R\|_{L^\infty}+\left\|\nabla \al\right\|_{L^\infty}+\left\|\nabla \bt\right\|_{L^\infty}\big)\,\left\|\big(\al,\bt\big)\right\|_{L^p}\,{\rm d}\t\,.
\end{equation}

Therefore, thanks to assumption \eqref{eq:ContCritPFinite} and Gr\"onwall's lemma, we deduce from \eqref{est:R_inf}, \eqref{est:DR_inf} and \eqref{est:a-b_p} above that
$$
\sup_{t\in[0,T[}\Big(\|R(t)\|_{L^\infty}\,+\,\|\nabla R(t)\|_{L^\infty}\,+\,\left\|\big(\al(t),\bt(t)\big)\right\|_{L^p}\Big)\,\leq\,C\,.
$$
Plugging this bound in \eqref{eq:ContGronPFinite2} gives, for all $t\in[0,T[\,$, the inequality
\begin{equation*}
\big\|\big(R(t),\al(t), \bt(t)\big)\big\|_{\B}\,\leq\,\big\| (R_0,\al_0, \bt_0) \big\|_{\B}
+C\int^t_0\big(\|\nabla\al\|_{L^\infty}+\|\nabla\bt\|_{L^\infty}+1\big)\big\|\big(R,\al,\bt\big)\big\|_{\B} {\rm d}\t,
\end{equation*}
which finally guarantees the desired property:
$$
\sup_{t\in[0,T[}\big\|\big(R(t),\al(t), \bt(t)\big)\big\|_{\B}\,\leq\,C\,.
$$
As already said at the beginning of the proof, by standard arguments and uniqueness of solutions, this latter inequality implies the result.
\end{proof}

\subsection{End of the proof} \label{ss:end}

Let us sum up what we have obtained so far, and close the proof to Theorems \ref{th:BesovWP} and \ref{th:cont-crit}.

\paragraph{Uniqueness.}
When $R=0$, uniqueness of solutions to system \eqref{i_eq:MHD-I} in the class $\B$, under conditions \eqref{i_eq:Lip} and \eqref{i_eq:p} was shown in \cite{MY}. This was achieved by using (somehow implicitly)
Els\"asser variables. This is enough, in view of the equivalence established in Theorem \ref{th:symm} above. Our case $R\neq 0$ can be treated with simple modifications, so, for the sake of brevity, we will not give details here.

\paragraph{Existence.}
Also for this, we invoke the result of \cite{MY}, valid in the case $R=0$, in particular for what concerns the construction of smooth approximate solutions to \eqref{i_eq:MHD-I} and their convergence to a true solution of that problem. Therefore, we limit ourselves to show only \tsl{a priori} bounds for regular solutions $(R,U,B)$ in Besov norms: those estimates are given by Proposition \ref{p:a-priori} for solutions $(R,\alpha,\beta)$ to system \eqref{eq:MHDabAPriori}, and simple algebraic manipulations yield analogous bounds for $(R,U,B)$.

\paragraph{Besov regularity of the time derivatives and of the pressure term.}

The Besov regularity $L^\infty_{T}(B^{s-1}_{p, r})$ of the time derivatives $(\d_tR,\d_tU,\d_tB)$ follows by applying the Leray projector $\P$ to equations \eqref{i_eq:MHD-I}
and using the Besov regularity for $(R,U,B)$ and product rules in Besov spaces. 
On the other hand, Proposition \ref{p:pressure} describes the $L^\infty_T(\B)$ regularity on the pressure terms $\nabla \pi_1=\nabla\pi_2$. Now, owing to the second part of Theorem \ref{th:symm}
and the fact that $\B$ is an algebra, we deduce that also $\nabla\Pi\in L^\infty_T(\B)$.

\paragraph{Time regularity properties.}
Once Besov regularity has been established for the pressure terms in \eqref{eq:MHDab},
time regularity properties for $(R,\al,\bt)$ are direct consequences of the solution's Besov regularity and standard results on transport equations in Besov spaces
(see \tsl{e.g.} Theorem \ref{th:transport} in the Appendix). From this, one immediately infers time regularity also for $U$ and $B$, as stated in Theorem \ref{th:BesovWP}.
At this point, passing to the Els\"asser formulation again and repeating the analysis explained above,
one can bootstrap the time regularity of the time derivatives $(\d_tR,\d_tU,\d_tB)$ and of the pressure term, gaining time continuity for those quantities.

\paragraph{Lower bound on the lifespan of the solutions, and continuation criterion.}
These are straightforward consequences of the corresponding results established in Propositions \ref{p:a-priori} and \ref{p:ContFinite}, together with the equivalence of the systems \eqref{i_eq:MHD-I} and \eqref{eq:MHDabAPriori} provided by Theorem \ref{th:symm}.


\section{Asymptotic results} \label{s:res-as}

Here we show that, in space dimension $d=2$, the quasi-homogeneous system \eqref{i_eq:MHD-I} can be rigorously derived from a density-dependent incompressible MHD system with Coriolis force,
in a certain asymptotic regime.
For the sake of completeness, we establish also a corresponding result for the viscous and resistive case, supplementing in this way the statements of \cite{Cobb-F} with a quantitative estimates
and precise rates of convergence.

\subsection{The primitive system} \label{ss:model}

Assume the spatial domain $\Omega$ to be either the whole space $\R^2$ or the torus $\T^2$. In $\R_+\times\Omega$, we consider the following \emph{non-homogeneous
incompressible MHD system} with Coriolis force:
\begin{equation}\label{MHD1}
\begin{cases}
\d_t\rho+\div\big(\rho\,u\big)\,=\,0 \\[1ex]
\partial_t\big(\rho u\big)+\D\big(\rho u \otimes u\big)+\dfrac{1}{\veps}\nabla \Pi+\dfrac{1}{\veps} \rho u^\perp\,=\,h(\veps)\D\big(\nu(\rho) \nabla u\big)+
\D\big(b \otimes b\big)-\nabla \dfrac{|b|^2}{2} \\[1ex]
\partial_t b + \D\big(u \otimes b\big) -\div\big(b \otimes u\big)\, =\,h(\veps) \nabla^\perp\big(\mu (\rho) \nabla\times b\big) \\[1ex]
\div u\,=\,\div b\, =\, 0\,.
\end{cases}
\end{equation}

In the previous system, $\rho\geq0$ represents the density of the fluid, $u\in\R^2$ its velocity field and $\Pi\in\R$ its (hydrodynamic) pressure field.
The vector $b\in\R^2$ represents an external magnetic field, which interacts with the fluid flow.
The viscosity coefficient $\nu$ and the resistivity coefficient $\mu$ are assumed to be continuous and non-degenerate: more precisely, they satisfy 
\begin{equation} \label{hyp:nu-mu}
\nu\,,\;\mu\,\in\,C^0(\R_+)\,,\quad\qquad\mbox{ with, }\;\forall\;\rho\geq0\,,\quad \nu(\rho)\geq\nu_*>0\quad\mbox{ and }\quad\mu(\rho)\geq\mu_*>0\,,
\end{equation}
for some positive real numbers $\nu_*$ and $\mu_*$. The function $h$ satisfies either $h\equiv 1$ (see Theorem \ref{THSt}) or 
$h\in C^0([0,1])$, with $h(\veps) > 0$ for $\veps>0$ and $h(0)=0$ (see Theorem \ref{th:InvConv}).

On the other hand, the term $\rho\mf C(u)\,:=\,\rho\,u^\perp$ is related to the Coriolis force acting on the fluid;
its presence in the equations is relevant for large-scale flows, typically at the surface of a planetary or stellar body. In that case, the rotation speed of the body is typically high compared to the other kinetic parameters. In view of these considerations,
$\rho \mf C(u)$ is penalised by multiplication by $\eps^{-1}$, where $\eps>0$ is a small parameter corresponding to the so-called \emph{Rossby number}. As is customary in this kind of problem, the pressure term is also penalised by the same factor; we refer to \cite{CDGG}, \cite{F-G} for details.

\medbreak
For any $\eps>0$, we consider initial data $\big(\rho_{0,\eps},u_{0,\eps},b_{0,\eps}\big)$ to the previous system,
and a corresponding global in time finite energy weak solution
$\big(\rho_\eps,u_\eps,b_\eps\big)$, in the sense specified by Definition \ref{d:weak} below.
In \cite{Cobb-F}, we have performed the limit for $\eps\ra0^+$ in equations \eqref{MHD1}
for general \emph{ill-prepared} initial data, both in the quasi-homogeneous and in the fully non-homogeneous regimes.

Here, we will focus on the former setting, namely the \emph{quasi-homogeneous case},
where the density is supposed to be a small perturbation of a constant state, say $\bar{\rho} = 1$ for simplicity of exposition (see condition \eqref{def:in-dens} below). Our goal is twofold: first of all,
we want to supplement the result of \cite{Cobb-F} with quantitative rates of convergence, highlighting the structure of the solutions to \eqref{MHD1} for small values of $\eps>0$.
In addition, we also want to consider the regime of vanishing viscosity and magnetic resistivity $h(\veps) \rightarrow 0^+$, which allows us to recover the ideal system \eqref{i_eq:MHD-I} in the limit.
Notice that the techniques of \cite{Cobb-F} do not apply to the study of this latter case.

\subsection{Initial data, finite energy weak solutions} \label{ss:initial}

We supplement system \eqref{MHD1} with general ill-prepared initial data. Let us be more precise, and start by considering the density functions: for any $0<\eps\leq1$, we take
\begin{equation} \label{def:in-dens}
\rho_{0, \veps}\, =\, 1\, +\, \veps\, r_{0, \veps}\,, \qquad\qquad \text{ with } \qquad \big(r_{0, \veps}\big)_{\veps>0}\,\subset\, \big(L^2 \cap L^\infty\big)(\Omega)\,. 
\end{equation}
Without loss of generality we can suppose that there exist two constants $0<\rho_*\leq \rho^*$ such that, for all $\eps>0$, one has
\begin{equation} \label{hyp:no-vacuum}
0\,<\,\rho_*\,\leq\,\rho_{0,\veps}\,\leq\,\rho^*\,.
\end{equation}
Therefore, by the first equation in \eqref{MHD1}, vacuum can never appear.
For this reason, we can work directly with the initial velocity fields $u_{0,\eps}$ (which are therefore always well-defined). We assume 
\begin{equation*}
\big(u_{0, \veps}\big)_{\veps > 0}\,\subset\, L^2(\Omega)\quad \qquad \mbox{ and }\qquad \qquad \div u_{0,\veps}\,=\,0\,.
\end{equation*}
Finally, for the magnetic fields, we choose initial data such that
$$
\big(b_{0, \veps}\big)_{\veps > 0}\,\subset\, L^2(\Omega)\quad \qquad \mbox{ and }\qquad \qquad \div b_{0,\veps}\,=\,0\,.
$$
Before going on, we remark that the assumption $\rho_{0, \veps}\, =\, 1\, +\, \veps\, r_{0, \veps}$ simplifies the equations very much. Indeed, since $\div u_\eps=0$
for any $\eps>0$, at any later time we still have $\rho_\veps\, =\, 1\, +\, \veps\, r_\veps$, with $r_\veps$ solving a linear transport equation
\begin{equation*} \label{eq:r_e}
\partial_t r_\veps + \D(r_\veps u_\veps) = 0\,,\qquad\qquad\qquad \big(r_\veps\big)_{|t = 0} = r_{0, \veps}\,.
\end{equation*}

We are now ready to introduce the definition of finite energy weak solution to system \eqref{MHD1}.
Here below, the notation $C_w \big([0, T] ; L^2(\Omega)\big)$ stands for the space of time-dependent functions, taking values in $L^2$, which are continuous with respect to the weak topology of that space.
\begin{defi}\label{d:weak}
Let $T > 0$ and $\veps\in\,]0,1]$ be fixed. Let $\big(\rho_{0,\veps}, u_{0,\veps}, b_{0,\veps}\big)$ be an initial datum fulfilling the previous assumptions.
We say that $\big(\rho, u, b\big)$ is a
\emph{finite energy weak solution} to system \eqref{MHD1} in $[0, T] \times \Omega$, related to the previous initial datum, if the following conditions are verified:
\begin{enumerate}[(i)]
\item $\rho \in L^\infty\big([0, T] \times \Omega\big)$ and $\rho \in C\big([0,T];L^q_{\rm loc}(\Omega)\big)$ for all $1 \leq q < +\infty$;
\item $u \in L^\infty\big([0,T];L^2(\Omega)\big) \cap C_w \big([0, T] ; L^2(\Omega)\big)$, with $\nabla u \in L^2\big([0,T];L^2(\Omega)\big)$;
\item $b \in L^\infty\big([0,T];L^2(\Omega)\big) \cap C_w \big([0, T] ; L^2(\Omega)\big)$, with $\nabla b\in L^2\big([0,T];L^2(\Omega)\big)$;
\item the mass equation is satisfied in the weak sense: for any $\phi \in \mathcal{D}\big([0, T] \times \Omega\big)$, one has
\begin{equation*}
\int_0^T \int_\Omega \bigg\{  \rho \partial_t \phi + \rho u \cdot \nabla \phi  \bigg\} \text{d}x \text{d}t + \int_\Omega \rho_{0,\veps} \phi(0,\cdot) \text{d}x =\int_\Omega \rho(T) \phi (T,\cdot) \dx \,;
\end{equation*}
\item the divergence-free conditions $\D(u) = \D(b) = 0$ are satisfied in $\mathcal{D}'\big(\,]0, T[ \,\times \Omega\big)$;
\item the momentum and magnetic field equations are satisfied in the weak sense: for any $\psi \in \mathcal{D}\big([0, T]\times \Omega ; \mathbb{R}^2\big)$ such that $\D(\psi) = 0$,
one has
\begin{multline*}
\int_0^T \int_\Omega \left\{ \rho u \cdot \partial_t \psi + \big(\rho u \otimes u - b \otimes b\big) : \nabla \psi 
- \frac{1}{\veps} \rho u^\perp \cdot \psi - h(\veps) \nu(\rho) \nabla u :\nabla \psi  \right\} \text{d}x \text{d}t \\
+ \int_\Omega \rho_{0,\veps} u_{0,\veps}\cdot \psi(0,\cdot) \text{d}x = \int_\Omega \rho(T) u(T) \cdot \psi(T,\cdot) \dx \,,
\end{multline*}
and for all $\z \in \mathcal{D}\big([0, T] \times \Omega;\R^2\big)$ one has
\begin{multline*} 
\int_0^T \int_\Omega \bigg\{  b \cdot \partial_t \z + \big(u \otimes b - b \otimes u\big):\nabla \z + h(\veps) \mu(\rho) \big(\nabla\times b\big)\big(\nabla\times \z\big)\bigg\} \dx \text{d}t  \\
+\int_\Omega b_{0,\veps} \cdot \z(0,\cdot) \dx = \int_\Omega b(T) \cdot \z(T,\cdot) \dx \,;
\end{multline*}
\item for almost every $t \in [0, T]\,$, the following energy balance holds true:
\begin{multline*} 
\int_\Omega \Big( \rho(t) |u(t)|^2 + |b(t)|^2 \Big) \dx +  \int_0^t \int_\Omega h(\veps) \Big( \nu(\rho) |\nabla u|^2 +  \mu(\rho) |\nabla\times b|^2 \Big) \dx \text{d}\tau\, \\
\leq\, \int_\Omega \Big( \rho_{0,\veps}|u_{0,\veps}|^2 + |b_{0,\veps}|^2 \Big) \dx\,.
\end{multline*}
\end{enumerate}
The solution $\big(\rho, u,b\big)$ is said to be \emph{global} if the above conditions hold for all $T > 0$.
\end{defi}

Before going on, we still have a few definitions and assumptions to give.

First of all, owing to the previous uniform bounds for the initial data, we deduce that, up to an extraction, one has the weak convergence properties
\begin{equation}\label{EQcv}
u_{0, \veps} \wtend U_0 \quad \text{ in }\; L^2(\Omega)\,, \qquad r_{0, \veps} \wtend^* R_0 \quad \text{ in }\; \big(L^2 \cap L^\infty\big)(\Omega)\,,\qquad
 b_{0, \veps} \wtend B_0 \quad \text{ in }\; L^2(\Omega)\,,
\end{equation}
for suitable functions $U_0$, $R_0$ and $B_0$ belonging to the respective functional spaces.

In order to derive quantitative estimates for solutions to \eqref{MHD1}, we need more precise assumptions on the viscosity and resistivity coefficients than the ones in \eqref{hyp:nu-mu} above.
We start by a definition (see \tsl{e.g.} Section 2.10 of \cite{BCD} for details).

\begin{defi} \label{d:mod-cont}
A \emph{modulus of continuity} is a continuous non-decreasing function $\s:[0,1]\,\longrightarrow\,\R_+$ such that $\s(0)=0$.

Given a modulus of continuity $\s$, the space $\mc{C}_\s(\R)$ is defined as the set of real-valued functions $a\in L^\infty(\R)$
such that
$$
|a|_{\mc{C}_\mu}\,:=\, \sup_{x \,\in \, \mathbb{R}} \,\, \sup_{|y| \,\in\,]0,1]}\frac{\left|a(x+y)\,-\,a(x)\right|}{\s(|y|)}\,<\,+\infty\,.
$$
We also define $\|a\|_{\mc{C}_\s}\,:=\,\|a\|_{L^\infty}\,+\,|a|_{\mc{C}_\s}$.

\end{defi}

In view of the previous definition, beside \eqref{hyp:nu-mu}, we also assume that there exists a modulus of continuity $\s$ such that
$$
\nu\,,\;\mu\,\in\,\mc C_\s(\R)\,.
$$
Strictly speaking, we only need this in a neighbourhood of $\rho = 1$, but we formulate the previous global assumption for simplicity of presentation.

\subsection{The viscous and resistive case} \label{ss:visc-res}

Under the hypotheses formulated in the previous section, and for the choice $h\equiv1$ in \eqref{MHD1},
it was shown in \cite{Cobb-F} that the limit dynamics for $\veps\ra0^+$ is described by the quasi-homogeneous MHD system
\begin{equation}\label{EQLim}
\begin{cases}
\partial_t R \,+\, \D (R\,U)\, =\, 0 \\[1ex]
\partial_t U \,+\, \D (U \otimes U)\, +\, \nabla \left(\Pi\, +\, \dfrac{|B|^2 }{2} \right)\, +\, R\,U^\perp\, =\, \nu(1) \,\Delta U \,+\, \D (B \otimes B) \\[1ex]
\partial_t B \,+\, \D (U \otimes B \,-\, B \otimes U)\, =\, \mu(1)\, \Delta B \\[1ex]
\D (U) \,=\, \D(B) \,=\, 0\,,
\end{cases}
\end{equation}
for some suitable pressure function $\Pi$ and with the triplet $\big(R_0,U_0,B_0\big)$, identified in \eqref{EQcv} above, as initial datum.
In addition, we have shown that equations \eqref{EQLim} are well-posed (this is Theorem 5.1 in \cite{Cobb-F}). More precisely, the following theorem holds true.
\begin{thm}\label{th:WPL}
Consider $0 < \beta < 1$ and $(R_0, U_0, B_0) \in H^{1 + \beta}(\Omega)\times H^1(\Omega) \times H^1(\Omega)$.

For that initial datum, there is exactly one solution $(R, U, B)$ of \eqref{EQLim} in the energy space,
that is such that $R \in L^\infty\big(\R_+;(L^2 \cap L^\infty)(\Omega)\big)$ and  $U, B \in L^\infty\big(\R_+;L^2(\Omega)\big)$, with $\nabla U, \nabla B\in L^2\big(\R_+;H^1(\Omega)\big)$.
Moreover, this unique solution satisfies $R \in L^\infty_{\rm loc}\big(\R_+;H^{1+\gamma}(\Omega)\big)$ for all $\gamma < \beta$, and
$U, B \in L^\infty_{\rm loc}\big(\R_+;H^1(\Omega)\big) \cap L^2_{\rm loc}\big(\R_+;H^2(\Omega)\big)$.
\end{thm}

With this theorem at hand, we are ready to state the quantitative convergence result in the viscous and resistive case.
\begin{thm}\label{THSt}
Let $h\equiv1$ in \eqref{MHD1} and $\nu$, $\mu$ be as in \eqref{hyp:nu-mu}. For a given modulus of continuity $\s$, assume in addition that $\nu, \mu \in C_\s(\mathbb{R})$.
Consider a sequence $\big(\rho_{0, \veps}, u_{0, \veps}, b_{0, \veps}\big)_{\veps > 0}$ of initial data satisfying the assumptions fixed
in Subsection \ref{ss:initial}, and let $\big(\rho_\veps,u_\veps,b_\veps\big)_{\veps>0}$ be a corresponding sequence of global in time finite energy weak solutions to system \eqref{MHD1}.
Define $M>0$ by
\begin{equation*} 
M\,:=\,\sup_{\eps>0}\| r_{0, \veps} \|_{L^\infty}\, +\,\sup_{\eps>0} \| u_{0, \veps} \|_{L^2}\, +\,\sup_{\eps>0} \| b_{0, \veps} \|_{L^2}\,.
\end{equation*}
Assume also that the triplet $\big(R_0, U_0, B_0\big)$, defined in \eqref{EQcv}, belongs to $H^{1 + \beta}(\Omega) \times H^1(\Omega) \times H^1(\Omega)$, for some $\beta\in\,]0,1[\,$,
and let $\big(R,U,B\big)$ be the corresponding unique solution to system \eqref{EQLim}, as given by Theorem \ref{th:WPL}.
Finally, set
$$
\de r_\eps\,:=\,r_\eps-R\,,\qquad \de u_\eps\,:=\,u_\eps-U\,,\qquad \de b_\eps\,:=\,b_\eps-B
$$
and, with analogous notation, $\de r_{0,\eps}\,:=\,r_{0,\eps}-R_0$,
$\de u_{0,\eps}\,:=\,u_{0,\eps}-U_0$
and $\de b_{0,\eps}\,:=\,b_{0,\eps}-B_0$.

Then, for all fixed times $T>0$, the following estimate holds true: for any $\eps>0$ and almost every $t\in[0,T]$,
\begin{multline}\label{EQRelIn}
\|\de r_\veps(t)\|_{L^2}^2 + \|\de u_\veps(t)\|_{L^2}^2 + \|\de b_\veps(t)\|_{L^2}^2+
\int_0^t\left\{ \left\| \nabla \de u_\veps\right\|_{L^2}^2+ \left\|\nabla\times\de b_\veps \right\|_{L^2}^2 \right\} \, {\rm d}\tau \\
\leq\, C\,\bigg\{\|\de r_{0, \veps}\|_{L^2}^2+\|\de u_{0, \veps}\|_{L^2}^2+\|\de b_{0, \veps}\|_{L^2}^2 +
\max\left\{\veps^2\,,\,\s^2(M \veps)\right\} \bigg\}\,,
\end{multline}
where the constant $C > 0$ depends on $T$, on the lower bounds $\nu_*$ and $\mu_*$ as well as on $|\nu|_{\mc C_\s}$ and $|\mu|_{\mc C_\s}$,
on the norms of the initial data $\| U_0 \|_{H^1}$, $\| B_0 \|_{H^1}$ and $\| R_0 \|_{H^{1+\beta}}$, and on $M$.
\end{thm}

\begin{rmk} \label{r:moduli}
In the simpler case when $\mu, \nu \in C^1(\mathbb{R}_+)$, the last summand in the brackets (of the right-hand side of \eqref{EQRelIn}) becomes $O(\veps^2)$.
Note that the convergence does not improve when \tsl{e.g.} $\mu$ and $\nu$ are constant near $\rho = 1$. As we will see below, this is mainly due to the fact that $\rho_\veps = 1 + O(\veps)$.
\end{rmk}

The previous theorem immediately yields the following corollary, which is in fact a convergence result: based on strong convergence of the initial data, we may deduce strong convergence of the solutions of \eqref{MHD1} to the solution of \eqref{EQLim}. This is very much in the spirit of \cite{KLS} and \cite{FLuN}.

\begin{cor}\label{COR1}
Under the same assumptions as in Theorem \ref{THSt} above, assume moreover that 
\begin{equation*}
\| r_{0, \veps} - R_0\|_{L^2}^2 + \| u_{0, \veps} - U_0\|_{L^2}^2 + \| b_{0, \veps} - B_0\|_{L^2}^2 \tend_{\veps \rightarrow 0^+} 0\,.
\end{equation*}

Then, for all fixed $T>0$, we have strong convergence of the solutions
\begin{equation*}
\big(u_\veps, b_\veps\big)\,\tend_{\veps \rightarrow 0^+}\, \big(U,B\big) \quad \text{ in }\; L^\infty_T(L^2) \cap L^2_T(H^1)\qquad\quad \text{ and }\qquad\quad
r_\veps\,\tend_{\veps \rightarrow 0^+}\,R \quad \text{ in }\; L^\infty_T(L^2)\,.
\end{equation*}
\end{cor}


\subsection{Vanishing viscosity and resistivity limit: the ideal case} \label{ss:ideal-res}

In this paragraph, we consider the vanishing viscosity and resistivity limit of system \eqref{MHD1}.
Namely, we now assume that
\begin{equation} \label{hyp:h}
h\,\in\,C^0([0,1])\,,\qquad\quad \mbox{ with }\qquad \forall\ \veps>0 \,, \ h(\veps)>0 \qquad \text{and} \qquad h(\veps)\,\tend_{\veps\ra0^+}\,0\,.
\end{equation}

The main challenge this new problem poses is that, because the elliptic parts of the equations vanish in \eqref{MHD1},
we can no longer rely on \emph{uniform} $L^2_{\text{loc}}(H^1)$ bounds for the velocity and the magnetic fields. We must therefore require additional regularity on the limit points,
which solve the quasi-homogeneous ideal MHD system
\begin{equation}\label{eq:IdMHD}
\begin{cases}
\partial_t R + \D(R U) = 0\\
\partial_t U + \D(U \otimes U - B \otimes B) + \nabla \pi + R U^\perp = 0\\
\partial_t B + \D (U \otimes B - B \otimes U) = 0\\
\D(U) = \D(B) = 0.
\end{cases}
\end{equation}
Notice that this system is exactly system \eqref{i_eq:MHD-I}, supplemented with the choice \eqref{eq:MatrixC} of the matrix $\mf C$ and with $\pi$ defined in \eqref{eq:p-P}.

As we have shown (see Theorem \ref{th:BesovWP} above for the precise statement), system \eqref{eq:IdMHD} is locally well-posed in critical Besov spaces
$B^{s}_{p, r}$, under conditions \eqref{i_eq:Lip} and \eqref{i_eq:p}, so that we can indeed consider the strong solutions we need. For simplicity, we are going to work on energy spaces
$H^s$, corresponding to the choice $r=p=2$.
We point out that Theorem \ref{th:InvConv} below holds on any time interval $[0, T]$
on which solutions of the limit problem exist.

\begin{thm}\label{th:InvConv}
Assume that assumptions \eqref{hyp:nu-mu} and \eqref{hyp:h} hold for the coefficients $\nu$, $\mu$ and $h$ in system \eqref{MHD1}.
Consider a sequence $\big(\rho_{0, \veps}, u_{0, \veps}, b_{0, \veps}\big)_{\veps > 0}$ of initial data satisfying the assumptions fixed
in Subsection \ref{ss:initial}, and let $\big(\rho_\veps,u_\veps,b_\veps\big)_{\veps>0}$ be a corresponding sequence of global in time finite energy weak solutions to system \eqref{MHD1}.
Define $M>0$ as above, \tsl{i.e.}
\begin{equation*} 
M\,:=\,\sup_{\eps>0}\| r_{0, \veps} \|_{L^\infty}\, +\,\sup_{\eps>0} \| u_{0, \veps} \|_{L^2}\, +\,\sup_{\eps>0} \| b_{0, \veps} \|_{L^2}\,.
\end{equation*}
Assume also that the triplet $\big(R_0, U_0, B_0\big)$, defined in \eqref{EQcv}, belongs to $\big(H^{s}(\Omega)\big)^3$, for some $s > 2$,
and let $\big(R,U,B\big) \in C^0\big([0,T^*[\,;H^s(\Omega)\big)^3$ be the corresponding unique strong solution to system \eqref{eq:IdMHD},
where $T^*>0$ denotes the maximal time of existence of that solution\footnote{According to Theorem \ref{th:BesovWP}, one has $T^* = T^*(\|R_0\|_{H^s}, \|U_0\|_{H^s}, \|B_0\|_{H^s})$
with possibly $T^*<+\infty$. The lifespan $T^*$ must also satisfy an explosion criterion, as implied by Proposition \ref{p:ContFinite}.}.
Finally, set
$$
\de r_\eps\,:=\,r_\eps-R\,,\qquad \de u_\eps\,:=\,u_\eps-U\,,\qquad \de b_\eps\,:=\,b_\eps-B
$$
and, with analogous notation, $\de r_{0,\eps}\,:=\,r_{0,\eps}-R_0$, $\de u_{0,\eps}\,:=\,u_{0,\eps}-U_0$ and $\de b_{0,\eps}\,:=\,b_{0,\eps}-B_0$.

Then, for all fixed times $T \in [0, T^*[\,$, the following estimate holds true: for any $\eps>0$ and almost every $t\in[0,T]$,
\begin{multline}\label{EQRelIn2}
\|\de r_\veps(t)\|_{L^2}^2 + \|\de u_\veps(t)\|_{L^2}^2 + \|\de b_\veps(t)\|_{L^2}^2+ h(\veps) \int_0^t\left\{ \left\| \nabla \de u_\veps\right\|_{L^2}^2+
\left\|\nabla\times\de b_\veps \right\|_{L^2}^2 \right\} \, {\rm d}\tau \\
\,\leq\, C\,\bigg\{\|\de r_{0, \veps}\|_{L^2}^2+\|\de u_{0, \veps}\|_{L^2}^2+\|\de b_{0, \veps}\|_{L^2}^2 + \veps^2 + h(\veps) \bigg\}\,,
\end{multline}
where the constant $C > 0$ depends on the lower bounds $\nu_*$ and $\mu_*$, as well as on the norm of the initial datum $\| (R_0, U_0, B_0) \|_{H^s}$, on $M$ and on $T$.
\end{thm}

Of course, a statement in the spirit of Corollary \ref{COR1} can be deduced also in this case.

\section{A relative entropy inequality for the primitive system} \label{s:entropy}

In this section we show that any finite energy weak solution to the non-homogeneous viscous and resistive MHD system \eqref{eq:MHD} satisfies a relative entropy inequality.

For simplicity, we assume the space domain $\Omega$ to be either the whole space $\R^2$ or the torus
$\T^2$, exactly as in Section \ref{s:res-as} above. However, more general domains may be allowed at this stage.

\subsection{Preliminaries}\label{ss:assumptions}
Let us collect here our main working assumptions, which will be assumed to hold true throughout the rest of this section.
They simply correspond to taking $\veps=1$ in Subsections \ref{ss:model} and \ref{ss:initial} above; however, for the reader's convenience, we will recall them here.

To begin with, the primitive system (recall \eqref{MHD1} and take $\veps=1$) is now the following non-homogeneous incompressible MHD system with Coriolis force:
\begin{equation}\label{eq:MHD}
\begin{cases}
\d_t\rho\,+\,\div\big(\rho\,u\big)\,=\,0 \\[1ex]
\partial_t\big(\rho u\big)\,+\,\D\big(\rho\, u \otimes u\big)\,+\,\nabla \Pi\,+\, \rho\, u^\perp\,=\,\D\big(\nu(\rho)\, \nabla u\big)\,+\,
\D\big(b \otimes b\big)\,-\,\dfrac{1}{2}\nabla |b|^2 \\[1ex]
\partial_t b\, +\, \D\big(u \otimes b\big) \,-\,\div\big(b \otimes u\big)\, =\, \nabla^\perp\big(\mu (\rho)\, \nabla\times b\big) \\[1ex]
\div u\,=\,\div b\, =\, 0\,.
\end{cases}
\end{equation}
The viscosity and resistivity coefficients, $\nu$ and $\mu$ respectively,  satisfy assumption \eqref{hyp:nu-mu}. 

System \eqref{eq:MHD} is supplemented with the initial datum $\big(\rho\,,\,u\,,\,b\big)_{|t=0}\,=\,\big(\rho_0\,,\,u_0\,,\,b_0\big)$ satisfying the following assumptions
(see also \cite{GB}, \cite{L}):
\begin{enumerate}[(a)]
 \item for the initial density, we require
\begin{equation*}
\rho_0\,\geq\,\rho_*\,>\,0\qquad\qquad \mbox{ and }\qquad\qquad \rho_0\,-\,1\,:=\,r_0\,\in\, L^2(\Omega)\cap L^\infty(\Omega)\,;
\end{equation*}
\item for the initial velocity and magnetic fields, we require
$$
u_0\,\in\, L^2(\Omega)\,,\quad b_{0}\,\in\, L^2(\Omega)\qquad\qquad\mbox{ and }\qquad\qquad \div u_0\,=\,\div b_0\,=\,0 \,.
$$
\end{enumerate}

Notice that the $L^2$ condition on $r_0$ is not really needed for the existence of weak solutions, but we assume it for later use in Subsection \ref{ss:result}.
Similarly, the assumption $\rho_0\geq\rho_*>0$ is enough in view of our study of Section \ref{s:rotation} (recall hypothesis \eqref{hyp:no-vacuum} above),
although it could be slightly relaxed (in the spirit of conditions (2.8) to (2.10) of \cite{L}, see also \cite{Cobb-F}) for formulating the relative entropy inequality.
Remark however that the weak solutions theory for system \eqref{eq:MHD} requires absence of vacuum (see \tsl{e.g.} \cite{GB}, \cite{DLe}).

For a given initial datum $(\rho_0, u_0, b_0)$ verifying the hypotheses above, we consider a global in time finite energy weak solution $(\rho,u,b)$ of \eqref{eq:MHD}.
Here, the definition of finite energy weak solution is the same as in Definition \ref{d:weak} above, with the choice $\veps=1$. In particular, we point out the following facts:
\begin{enumerate}[(i)]
 \item the weak formulation of the momentum equation becomes: for any $\psi \in \mathcal{D}\big([0, T]\times \Omega ; \mathbb{R}^2\big)$ such that $\D(\psi) = 0$,
one has
\begin{multline}\label{Momentum}
\int_0^T \int_\Omega \left\{ \rho u \cdot \partial_t \psi + \big(\rho u \otimes u - b \otimes b\big) : \nabla \psi 
-\rho u^\perp \cdot \psi - \nu(\rho) \nabla u :\nabla \psi  \right\} \text{d}x \text{d}t \\
+ \int_\Omega \rho_0 u_0\cdot \psi(0,\cdot) \text{d}x = \int_\Omega \rho(T) u(T) \cdot \psi(T,\cdot) \dx \,;
\end{multline}
\item the weak formulation of the magnetic field equation now reads: for all $\z \in \mathcal{D}\big([0, T] \times \Omega;\R^2\big)$ one has
\begin{multline} \label{eq:magn}
\int_0^T \int_\Omega \bigg\{  b \cdot \partial_t \z + \big(u \otimes b - b \otimes u\big):\nabla \z + \mu(\rho) \big(\nabla\times b\big)\big(\nabla\times \z\big)\bigg\} \dx \text{d}t  \\
+\int_\Omega b_0 \cdot \z(0,\cdot) \dx = \int_\Omega b(T) \cdot \z(T,\cdot) \dx \,;
\end{multline}
\item the energy inequality becomes:
for almost every $t \in [0, T]\,$, the following energy balance holds true:
\begin{multline} \label{EN}
\int_\Omega \Big( \rho(t) |u(t)|^2 + |b(t)|^2 \Big) \dx +  \int_0^t \int_\Omega \Big( \nu(\rho) |\nabla u|^2 +  \mu(\rho) |\nabla\times b|^2 \Big) \dx \text{d}\tau\, \\
\leq\, \int_\Omega \Big( \rho_{0}|u_{0}|^2 + |b_{0}|^2 \Big) \dx\,.
\end{multline}
\end{enumerate}

Existence of global in time finite energy weak solutions (in the previous sense) to system \eqref{eq:MHD} has been shown by Gerbeau and Le Bris in \cite{GB}
in a bounded domain of $\mathbb{R}^3$; see also \cite{DLe} for related results and additional references.
The proof of \cite{GB} can be extended to the cases $\mathbb{R}^2$ or $\mathbb{T}^2$ (which are relevant for our study) with standard modifications.

We point out that, in \cite{GB}, the authors resort to P.-L. Lion's theory \cite{L} for non-homogeneous fluids with density-dependent viscosities but with no magnetic field.
In that case, the initial density is even allowed to vanish, under some suitable non-degeneracy condition; on the contrary, the result of \cite{GB} holds
only for fluids with non-vanishing initial densities, namely $\rho_0>0$ (see Remark 3.4 of \cite{GB} for more comments about this issue). In view of assumption \eqref{hyp:no-vacuum},
the result of \cite{GB} is all what we need for the present study.

Before going on, some remarks are in order.
\begin{rmk}
Observe that, owing to the divergence-free condition on the magnetic field and Plancharel's theorem, when the space dimension is $d=2$ one has
\begin{equation*}
\big\| \nabla b \big\|_{L^2} = \big\| -i \xi \what{b}(\xi) \big\|_{L^2} = \big\| -i \xi^\perp \cdot \what{b}(\xi) \big\|_{L^2} = \big\| \nabla \times b \big\|_{L^2}\,.
\end{equation*}
Therefore, inequality \eqref{EN} above yields an $L^2$ bound for the full gradient of $b$.
\end{rmk}

\begin{rmk} \label{r:weak}
When vacuum is permitted, or even when $\rho_0>0$ (without a uniform lower bound), the (weak) time continuity of the velocity field is no longer true in general. On the other hand,
one can show that $\mbb P (\rho u)$ is (weakly) continuous in time, where $\mbb P$ is the Leray projector onto the space of divergence-free vector fields (see Theorem 2.2 of \cite{L}
and Remark 3.1 of \cite{GB} in this respect). 
\end{rmk}

To conclude this part, we point out that, as the fluid density $\rho$ is simply transported by the divergence-free velocity field $u$, for all $t\geq0$ we get
\begin{equation} \label{est:rho_L^p}
\forall\;p\,\in\,[2,+\infty]\,,\qquad
 \left\|\rho(t)-1\right\|_{L^p}\,=\,\left\|\rho_0-1\right\|_{L^p}\,,\qquad\qquad \left\|\rho(t)\right\|_{L^\infty}\,=\,\left\|\rho_0\right\|_{L^\infty}\,.
\end{equation}

\subsection{Stating the relative entropy inequality} \label{ss:result}
Following what done in Subsection \ref{ss:initial}, we set 
\begin{equation} \label{eq:rho-r}
r\,:=\,\rho\,-\,1\qquad\qquad\mbox{ and }\qquad\qquad r_0\,:=\,\rho_0\,-\,1\,,
\end{equation}
and we notice that, owing to the first equation in \eqref{eq:MHD} and the condition $\div(u)=0$, one has
\begin{equation} \label{eq:r}
\d_tr\,+\,\div\big(r\,u\big)\,=\,0\,,\qquad\qquad r_{|t=0}\,=\,r_0\,.
\end{equation}
Since this relation is completely analogous to the mass conservation equation, 
throughout all this section we will equivalently speak of solutions $(\rho,u,b)$ and $(r,u,b)$ to the MHD system \eqref{eq:MHD},
implying that $r$ and $\rho$ are linked through \eqref{eq:rho-r}.

Next, we define the \emph{relative entropy} of a solution $(r,u,b)$ of the non-homogeneous system \eqref{eq:MHD} with respect to a triplet $(R,U,B)$ of (say) smooth, compactly supported functions in $\R_+\times\Omega$, to be the quantity 
\begin{equation} \label{def:entropy}
\mc E\Big([r,u,b]\,\Big|\,[R,U,B]\Big)\,:=\,\frac{1}{2} \int_\Omega \bigg\{\rho\, \big|u-U\big|^2\, +\, \big|b-B\big|^2\, +\, \big|r-R\big|^2 \bigg\}\,\dx\,.
\end{equation}
We can then formulate the following \textit{relative entropy inequality}: for almost every $T>0$,
\begin{align}
&\mc E\Big([r,u,b]\,\Big|\,[R,U,B]\Big)(T)\,+\,\int^T_0\!\!\!\int_\Omega\left(\nu(\rho)\,\big|\nabla(u-U)\big|^2\,+\,\mu(\rho)\,\big|\nabla\times(b-B)\big|^2\right)\dx\,\dt \label{est:rel-entropy} \\
&\qquad\qquad\qquad\qquad\qquad\leq\,\mc E\Big([r_0,u_0,b_0]\,\Big|\,[R(0),U(0),B(0)]\Big)\,+\,\int^T_0\mc R\big(r,u,b;R,U,B\big)\,\dt\,, \nonumber
\end{align}
where we have defined
\begin{align}
\mc R\big(r,u,b;R,U,B\big)\,:=&\,-\int_\Omega\rho\left(\d_tU+ (u\cdot\nabla) U+U^\perp\right)\cdot\de u\dx-\int_\Omega\big(\d_tB+ (u\cdot\nabla) B\big)\cdot\de b\dx \label{def:R}  \\
&\;-\int_\Omega\big(\d_tR+u\cdot\nabla R\big)\cdot\de r\dx+\int_\Omega\Big(\big(  b \cdot\nabla \big) U \cdot\de b+\big( b\cdot\nabla \big) B\cdot\de u\Big)\dx \nonumber \\
&\;-\int_\Omega\nu(\rho)\,\nabla U:\nabla \de u\dx-\int_\Omega\mu(\rho)\,\big(\nabla\times B\big)\,\big(\nabla\times \de b\big)\dx\,. \nonumber
\end{align}

The main goal of this section is to prove that any global in time finite energy weak solution to \eqref{eq:MHD} satisfies the previous relative entropy inequality,
as established in the next statement. Its proof is postponed to the following subsection.

\begin{thm} \label{t:rel-en}
Let $(r_0, u_0, b_0) \in (L^\infty \cap L^2)(\Omega) \times L^2(\Omega) \times L^2(\Omega)$ be an initial datum for problem \eqref{eq:MHD}, and let $(r, u, b)$ be a global in time
finite energy solution related to that initial datum.

Then, for any triplet $(R, U, B)$ of functions enjoying the following regularity properties, namely
\begin{enumerate}[(1)]
\item $R \in W^{1,1}_{\rm loc}\big(\R_+;L^2(\Omega)\big)$, with $\nabla R\in L^2_{\rm loc}\big(\R_+;L^q(\Omega)\big)$ for some $q>2$,
\item $U, B\,\in W^{1, 1}_{\rm loc}\big(\R_+;L^2(\Omega)\big)$, with $\nabla U, \nabla B\, \in  L^2_{\rm loc}\big(\R_+;(L^2 \cap L^{q})(\Omega)\big)$ for some $q > 2$,
\item $\div U\,=\,\div B\,=\,0$ almost everywhere in $\R_+\times\Omega$,
\end{enumerate}
the relative entropy inequality \eqref{est:rel-entropy} holds for almost every $T>0$.
\end{thm}

Before going on, let us list a series of possible extensions of the previous result.
First of all, Theorem \ref{t:rel-en} can be easily formulated also in the higher dimensional case, provided the Coriolis term vanishes or is changed in an appropriate
(physically relevant) way. In addition, the technique of the proof can be employed to handle more complex geometries of the space domain,
encoding different boundary conditions (\tsl{e.g.} no-slip, or complete slip boundary conditions). 
Finally, as it will appear clearly from the proof, the regularity of the triplet $(R,U,B)$ can also be somehow relaxed, and different integrability hypotheses may be imposed.

However, we refrain from treating such situations in full generality, and we limit ourselves to state and prove the result in the case
which is of interest for our applications.

\subsection{Proof of the relative entropy inequality} \label{ss:proof}

In order to prove Theorem \ref{t:rel-en}, we proceed as in \cite{FJN}, where the authors prove similar inequalities for compressible Navier-Stokes equations.
For simplicity, let us consider for a while a triplet $(R, U, B)$ of smooth functions such that:
\begin{enumerate}[(i)]
 \item $R\,\in\,\mc D\big(\R_+\times\Omega\big)$;
 \item $U$ and $B$ belong to $\mc D\big(\R_+\times\Omega;\R^2\big)$ and are such that $\;\div U\,=\,\div B\,=\,0$.
\end{enumerate}
Let $T>0$ be such that the support of $R$, $U$ and $B$ is included in $[0,T]\times\Omega$.


First of all, using $\psi = U$ as a test function in the weak form \eqref{Momentum} of the momentum equation, after noting that $u^\perp\cdot U\,=\,-u\cdot U^\perp$, we find
\begin{align}
&\int_\Omega \rho(T) u(T) \cdot U(T) \dx= \int_\Omega \rho_{0} u_{0} \cdot U(0) \dx + \int_0^T\!\!\!\int_\Omega \rho u \cdot \partial_tU \dx \dt  \label{eq:uU} \\   
&\quad+\int_0^T\!\!\! \int_\Omega \big( \rho u \otimes u - b \otimes b\big) : \nabla U \dx \dt+\int_0^T\!\!\! \int_\Omega \rho u \cdot U^\perp \dx \dt - \int_0^T\!\!\! \int_\Omega \nu(\rho) \nabla u : \nabla U \dx \dt\,. \nonumber
\end{align}
Next, testing the mass equation in \eqref{eq:MHD} against $|U|^2/2$ yields
\begin{align} 
&\frac{1}{2}\! \int_\Omega \rho(T) |U(T)|^2 \dx = \frac{1}{2}\! \int_\Omega \rho_{0} |U(0)|^2 \dx + \int_0^T\!\!\! \int_\Omega \rho \, U \cdot \partial_tU \dx \dt +
\frac{1}{2}\! \int_0^T\!\!\! \int_\Omega \rho \, u\cdot \nabla |U|^2 \dx \dt \label{eq:UU} \\
&\qquad\qquad\qquad\qquad= \frac{1}{2}\! \int_\Omega \rho_{0} |U(0)|^2 \dx + \int_0^T\!\!\! \int_\Omega \rho \, U \cdot \partial_tU \dx \dt +
\int_0^T\!\!\! \int_\Omega \rho \, u\otimes U:\nabla U \dx \dt\,. \nonumber
\end{align}
Recall that the energy inequality \eqref{EN} reads
\begin{multline}\label{eq:ub}
\frac{1}{2} \int_\Omega \bigg\{ \rho (T) |u (T)|^2 + |b(T)|^2 \bigg\} \dx \leq \frac{1}{2} \int_\Omega \bigg\{ \rho_{0} |u_{0}|^2 + |b_{0}|^2  \bigg\} \dx \\
- \int_0^T \int_\Omega \bigg\{ \nu(\rho) |\nabla u|^2 + \mu(\rho) |\nabla\times b|^2 \bigg\} \dx \dt\,.
\end{multline}

Now, let us deal with the magnetic field. We start by using $\z = B$ as a test function in the magnetic field equation \eqref{eq:magn}: we get
\begin{multline}\label{eq:bB}
\int_\Omega b(T) \cdot B(T) \dx  =\int_\Omega b_{0} \cdot B(0) \dx + \int_0^T \int_\Omega b \cdot \partial_t B \dx \dt \\
 +\int_0^T \int_\Omega\big\{ u \otimes b - b \otimes u \big\} :\nabla B \dx \dt - \int_0^T\int_\Omega \mu(\rho) \big(\nabla\times b\big)\big(\nabla\times B\big) \dt \dt\,.
\end{multline}
By analogy with what we have done for the velocity field above, we now use $\phi = |B|^2/2$ as a test function in the (trivial) transport equation $\partial_t1 + \div(1\,u) = 0$:
we find
\begin{align} 
\frac{1}{2} \int_\Omega |B(T)|^2 \dx &=  \frac{1}{2} \int_\Omega |B(0)|^2 \dx + \int_0^T \int_\Omega B \cdot \partial_t B \dx + \frac{1}{2} \int_0^T \int_\Omega u \cdot \nabla |B|^2 \dx \dt \label{eq:BB} \\
&=  \frac{1}{2} \int_\Omega |B(0)|^2 \dx + \int_0^T \int_\Omega B \cdot \partial_t B \dx + \int_0^T \int_\Omega u\otimes B:\nabla B \dx \dt\,. \nonumber
\end{align}

Finally, we take care of the density oscillation functions $r$ and $R$. Testing equation \eqref{eq:r} against the smooth $R$ gives
\begin{equation}\label{eq:rR}
\int_\Omega r(T) R(T) \dx = \int_\Omega r_0 R(0) \dx+\int_0^T \int_\Omega r \partial_tR \dx \dt  + \int_0^T \int_\Omega r u\cdot \nabla R \dx \dt\,.
\end{equation}
Using again \eqref{eq:r} and the fact that $\div u=0$, we deduce that the $L^2$ norm of $r$ must be preserved in time (this is the same property as \eqref{est:rho_L^p} above). In other words,
\begin{equation}\label{eq:rr}
\int_\Omega |r(T)|^2 \dx = \int_\Omega |r_{0}|^2 \dx\,.
\end{equation}
On the other hand, repeating the computations which led to \eqref{eq:BB}, we obtain
\begin{equation} \label{eq:RR}
\frac{1}{2} \int_\Omega |R(T)|^2 \dx = \frac{1}{2} \int |R(0)|^2 \dx + \int_0^T \int_\Omega R \partial_t R \dx \dt + \int^T_0\int_\Omega Ru\cdot\nabla R\dx \dt\,.
\end{equation}

At this point, for notational convenience let us define
$$
\de r\,:=\,r-R\,,\qquad \de u\,:=\,u-U\,,\qquad \de b\,:=\,b-B\,.
$$
Putting relations \eqref{eq:uU}, \eqref{eq:UU}, \eqref{eq:ub}, \eqref{eq:bB}, \eqref{eq:BB}, \eqref{eq:rR}, \eqref{eq:rr} and \eqref{eq:RR} all together, we eventually find
\begin{align}
&\frac{1}{2} \int_\Omega \bigg\{\rho(T)\, |\de u(T)|^2\, +\, |\de b(T)|^2\, +\, |\de r(T)|^2 \bigg\}\,\leq\,
\frac{1}{2}\int_\Omega \bigg\{ \rho_{0}\,|\de u_{0}  |^2\,+\,|\de b_{0}|^2\, +\, |\de r_{0}|^2\bigg\} \label{EQLong} \\
&\qquad\quad - \int_0^T\!\!\!\int_\Omega \nu(\rho)\,\nabla u:\nabla \de u\,-\, \int_0^T\!\!\!\int_\Omega \mu(\rho)\,\big(\nabla\times b\big)\,\big(\nabla\times\de b\big)\,-\,
\int_0^T\!\!\!\int_\Omega \rho\, u \cdot U^\perp \nonumber\\
&\qquad\quad-\int_0^T\!\!\!\int_\Omega \bigg\{\rho\,\de u \cdot \partial_tU\, +\, \de b \cdot \partial_t B\, +\, \de r\, \partial_tR \bigg\}\,-\,
\int_0^T\!\!\!\int_\Omega\bigg\{\rho\, u\otimes\de u:\nabla U\,+\,\de r\, u\cdot\nabla R\bigg\} \nonumber\\
&\qquad\quad+ \int_0^T \int_\Omega \bigg\{-\,u\otimes\de b:\nabla B\, +\, b \otimes b : \nabla U\, +\,b \otimes u : \nabla B\bigg\}\,. \nonumber
\end{align}
Next, we remark that we can write\footnote{Notice that we have set $f\otimes g:\nabla h=\sum_{j,k}f_jg_k\d_jh_k$; this corresponds to
the agreement that $[\nabla h]_{ij} = \partial_i h_j$ is the transpose matrix of the differential $Dh$ of $h$.} 
$$
\rho\,u\otimes\de u:\nabla U\,=\,\rho\,\big(u\cdot\nabla \big) U \cdot \de u\qquad\mbox{ and }\qquad u\otimes\de b:\nabla B\,=\,\big(u\cdot\nabla \big) B \cdot\de b\,,
$$
and that, by orthogonality, we have $u\cdot U^\perp\,=\,\de u\cdot U^\perp$.
Therefore, the right-hand side of \eqref{EQLong} can be recasted as
\begin{align*}
&\mc E\Big([r_0,u_0,b_0]\,\Big|\,[R(0),U(0),B(0)]\Big) \\
&-\int^T_0\!\!\!\int_\Omega\rho\left(\d_tU+ (u\cdot\nabla) U+U^\perp\right)\cdot\de u-\int^T_0\!\!\!\int_\Omega\big(\d_tB+ (u\cdot\nabla) B\big)\cdot\de b
-\int^T_0\!\!\!\int_\Omega\big(\d_tR+u\cdot\nabla R\big)\cdot\de r \\
&-\int^T_0\!\!\!\int_\Omega\nu(\rho)\,\nabla u:\nabla \de u-\int^T_0\!\!\!\int_\Omega\mu(\rho)\,\big(\nabla\times b\big)\,\big(\nabla\times \de b\big)+
\int^T_0\!\!\!\int_\Omega\bigg\{b \otimes b : \nabla U\, +\,b \otimes u : \nabla B\bigg\}\,.
\end{align*}
Let us focus on the last two terms for a while. Simple manipulations show that
\begin{align*}
b\otimes b:\nabla U\,&=\,\big(b\cdot\nabla \big) U\cdot b\,=\,\big(b\cdot\nabla\big) U \cdot \de b\,+\,\big(b\cdot\nabla \big) U\cdot B \\
b\otimes u:\nabla B\,&=\,\big(b\cdot\nabla \big) B\cdot u\,=\,\big(b\cdot\nabla \big) B\cdot\de u\,+\,\big(b\cdot\nabla \big) B \cdot U\,.
\end{align*}
Observe that $\big(b\cdot\nabla \big) U\cdot B\,+\,\big(b\cdot\nabla \big) B\cdot U\,=\,b\cdot\nabla\big( B\cdot U\big)$. Therefore, owing to the the divergence-free condition on $b$,
the previous relations imply that
$$
\int^T_0\!\!\!\int_\Omega\bigg\{b \otimes b : \nabla U\, +\,b \otimes u : \nabla B\bigg\}\,=\,\int^T_0\!\!\!\int_\Omega\bigg\{\big(b\cdot\nabla\big) U\cdot\de b\,+\,
\big(b\cdot\nabla\big) B\cdot\de u\bigg\}\,.
$$
This completes the proof of \eqref{est:rel-entropy} for smooth compactly supported $(R,U,B)$.

\medbreak
In order to get the result for triplets $(R,U,B)$ having the regularity stated in Theorem \ref{t:rel-en}, we argue by density.

Firstly, we observe that, for the left-hand side of \eqref{est:rel-entropy} to be well-defined, it is enough to have
\begin{equation} \label{cond:1}
|U|^2\,+\,|B|^2\,+\,R^2\,\in\,L^\infty_{T}(L^1)\qquad\mbox{ and }\qquad \nabla U\,,\ \nabla B\;\in\, L^2_T(L^2)\,.
\end{equation}
Next, we consider each term appearing in the definition \eqref{def:R} of $\mc R$. In view of a possible application of Gr\"onwall lemma, in the first three terms it is natural to put
$\sqrt{\rho}\, \de u$, $\de B$ and $\de r$ in $L^\infty_T(L^2)$, so that, since $\rho\in L^\infty(\R_+\times\Omega)$, one needs
$$
\d_tU+(u\cdot\nabla) U+U^\perp\,,\quad \d_t B+(u\cdot\nabla) B\,,\quad \d_tR+(u\cdot\nabla) R\qquad\mbox{ to belong to }\qquad L^1_T(L^2)\,.
$$
Owing to the regularity of $u\in L^2_T(H^1)$ and Sobolev embeddings, it is enough to require, besides the above conditions \eqref{cond:1}, also the conditions
\begin{equation} \label{cond:2}
\d_tU\,, \ \d_tB\,,\ \d_tR\;\in\,L^1_T(L^2)\qquad \mbox{ and }\qquad \nabla U\,,\ \nabla B\,,\ \nabla R\,\in\,L^2_T(L^q)\,,
\end{equation}
for some $q>2$. Conditions \eqref{cond:1} are enough also for the last two terms appearing in \eqref{def:R} to be well-defined. Finally, for the last integral appearing in the second line of \eqref{def:R},
we remark that, by Gagliardo-Nirenberg inequality of Lemma \ref{GN}, we deduce the property $b\in L^{2p/(p-2)}_T(L^p)$ for any $2\leq p<+\infty$, and the same actually holds true also for 
$\de u$ and $\de b$. Taking $p=4$, we infer that $b$, $\de u$ and $\de b$ belong to $L^4_T(L^4)$, so that conditions \eqref{cond:1} are also enough for treating
those terms.

The proof to Theorem \ref{t:rel-en} is now completed.

\section{Derivation of quasi-homogeneous MHD systems in $2$-D} \label{s:rotation}

This section is devoted to the proof of Theorems \ref{THSt} and \ref{th:InvConv}, concerning the rigorous derivation of systems \eqref{EQLim} and \eqref{eq:IdMHD} respectively,
and carried out respectively in Subsections \ref{ss:deriv-NS} and \ref{ss:VanVisc} below.
The main tool will be the relative entropy inequality for the primitive system, proved in the previous section.

Notice that the proof in the viscous and resistive case is actually more delicate than in the ideal case, inasmuch as one disposes of less regularity on the limit points $(R, U, B)$, which act as the smooth functions to be used in the relative entropy inequality \eqref{def:entropy}.

\subsection{Derivation of the viscous and resistive system} \label{ss:deriv-NS}

In this subsection we carry out the proof of Theorem \ref{THSt}.

To begin with, we need the following result from \cite{Cobb-F} (see Theorem 5.3 therein), which makes quantitative the regularity properties stated in Theorem \ref{th:WPL}.

\begin{prop}\label{p:LimitEstimates}
Let $\beta\in\,]0,1[\,$ and $(R_0, U_0, B_0) \in H^{1 + \beta} \times H^1 \times H^1$ be a set of initial data, and consider $(R, U, B)$ the corresponding unique solution to system \eqref{EQLim}
provided by Theorem \ref{th:WPL}. Then the following estimates hold true:
\begin{enumerate}[(i)]
\item the basic \emph{energy inequality}: for any $t \geq 0$, one has
\[ 
\frac{1}{2} \big\| \big( U(t), B(t) \big) \big\|_{L^2} + \int_0^t \int_\Omega \Big( \nu(1) \| \nabla U \|_{L^2}^2 + \mu(1) \| \nabla B \|_{L^2} \Big) \dx \, {\rm d \tau}
\leq \frac{1}{2} \big\| (U_0, B_0) \big\|_{L^2}\,;
\] 
\item the basic \emph{transport estimates}: for any $p\in[2,+\infty]$ and any $t\geq0$, one has
$$
\| R(t) \|_{L^p} = \| R_0 \|_{L^p}\,;
$$
\item for any exponent $h > 0$, there is a constant $C = C\big( \| R_0 \|_{L^2 \cap L^\infty}, \| (U_0, B_0) \|_{H^1}, h \big) > 0$ such that, for any time $T > 0$, one has
\[ 
\big\|  \nabla U \big\|_{L^\infty_T(L^2)} + \big\|  ( \Delta U, \partial_t U) \big\|_{L^2_T(L^2)} + \big\|  \nabla B \big\|_{L^\infty_T(L^2)} + 
\big\|  ( \Delta B, \partial_t B) \big\|_{L^2_T(L^2)} \leq C \left( 1 + T^h \right)\,;
\] 
\item for any $0 < \gamma < \beta$, there exists $C=C(\g,\beta)>0$ such that, any fixed time $T > 0$, one has
\[ 
\| R \|_{L^\infty_T(H^\gamma)} \leq C \| R_0 \|_{H^\beta} \exp \left\{ C \left( \int_0^T \| \nabla U \|_{H^{1}} \dt \right)^2 \right\}\,.
\] 
\end{enumerate}
\end{prop}

Now we can proceed to the proof of Theorem \ref{THSt}. As already mentioned, the main tool here is the relative entropy inequality established in Theorem \ref{t:rel-en} above.

\subsubsection{Using the relative entropy inequality} \label{sss:rel-entr-rot}

Owing to the properties stated in Theorem \ref{th:WPL} and standard product rules in Sobolev spaces (see Proposition \ref{p:op} below),
one can see that the solution $\big(R,U,B\big)$ to the limit problem \eqref{EQLim} meets the regularity requirements of Theorem \ref{t:rel-en}. Thus, it
can be employed as a test function in the relative entropy functional \eqref{def:entropy}.

Therefore, the relative entropy inequality \eqref{est:rel-entropy} yields, for all $T>0$, the estimate
\begin{align*}
&\mc E\Big([r_\eps,u_\eps,b_\eps]\,\Big|\,[R,U,B]\Big)(T)\,+\,\int^T_0\!\!\!\int_\Omega\left(\nu(\rho_\eps)\,\big|\nabla\de u_\eps\big|^2\,+\,
\mu(\rho_\eps)\,\big|\nabla\times \de b_\eps\big|^2\right)\dx\,\dt \\ 
&\qquad\qquad\qquad\qquad\qquad\qquad\qquad\qquad\leq\,\mc E\Big([r_{0,\eps},u_{0,\eps},b_{0,\eps}]\,\Big|\,[R_0,U_0,B_0]\Big)\,+\,\int^T_0\mc R_\eps\,\dt\,, 
\end{align*}
where the reminder term $\mc R_\eps\,:=\,\mc R\big(r_\eps,u_\eps,b_\eps;R,U,B\big)$ is defined by the formula
\begin{align*}
\mc R_\eps\,=\,\sum_{j=1}^7I_j\,=&\,
-\int_\Omega\rho_\eps\left(\d_tU+ (u_\eps\cdot\nabla) U+\frac{1}{\eps}U^\perp\right)\cdot\de u_\eps\,\dx\,-\int_\Omega\big(\d_tB+ (u_\eps\cdot\nabla) B\big)\cdot\de b_\eps\,\dx \\
&\,-\int_\Omega\big(\d_tR+u_\eps\cdot\nabla R\big)\cdot\de r_\eps\,\dx\,+\int_\Omega\big(b_\eps\cdot\nabla \big) U \cdot\de b_\eps\,\dx\,+\int_\Omega\big(b_\eps\cdot\nabla \big) B \cdot\de u_\eps\,\dx \\
&\,-\int_\Omega\nu(\rho_\eps)\,\nabla U:\nabla \de u_\eps\,\dx\,-\int_\Omega\mu(\rho_\eps)\,\big(\nabla\times B\big)\,\big(\nabla\times \de b_\eps\big)\,\dx\,. 
\end{align*}

At this point, we use the fact that $\big(R,U,B\big)$ is a strong solution of the limit system \eqref{EQLim}. We start by focusing on the density term $R$: using the first relation in \eqref{EQLim} yields
\begin{equation} \label{eq:R-term}
\d_tR+u_\eps\cdot\nabla R\,=\,-U\cdot\nabla R+u_\eps\cdot\nabla R\,=\,\de u_\eps\cdot\nabla R\,.
\end{equation}

Similarly, for the magnetic field we get
$$
\d_tB+ (u_\eps\cdot\nabla) B\,=\, (\du \cdot\nabla) B+(B\cdot\nabla) U+\mu(1)\Delta B\,.
$$
Observe that $\Delta B\,=\,\nabla^\perp\nabla\times B$, owing to the divergence-free condition on $B$. Therefore, putting together $I_2$, $I_4$ and $I_7$, thanks
to the previous equality we can write
\begin{multline} \label{eq:B-term}
I_2\,+\,I_4\,+\,I_7\,
=\,\int_\Omega\big(\de b_\eps\cdot\nabla \big) U \cdot\de b_\eps\,-  \int_\Omega (\du \cdot \nabla)B \cdot \db \\
- \,\int_\Omega\big(\mu(\rho_\eps)-\mu(1)\big)\big(\nabla\times B\big)\,\big(\nabla\times \de b_\eps\big)\,,
\end{multline}
where we have also used an integration by parts for the term presenting $\Delta B$.

It remains us to deal with the velocity field. First of all, because of the decomposition $\rho_\eps\,=\,1+\eps r_\eps$, the term $I_1$ can be recasted as
\begin{align}
I_1\,&=\,-\int_\Omega\left(\d_tU+ (u_\eps\cdot \nabla) U+\frac{1}{\eps}U^\perp\right)\cdot\de u_\eps-\eps\int_\Omega r_\eps\left(\d_tU+(u_\eps\cdot\nabla) U+\frac{1}{\eps}U^\perp\right)\cdot\de u_\eps 
\label{eq:I_1} \\
&=\,-\int_\Omega\big(\d_tU+ (u_\eps\cdot\nabla) U\big)\cdot\de u_\eps\,-\,\eps\int_\Omega r_\eps\left(\d_tU+ (u_\eps\cdot\nabla) U+\frac{1}{\eps}U^\perp\right)\cdot\de u_\eps \nonumber \\
&=\,I_{11}+I_{12}\,. \nonumber
\end{align}
Notice that, in passing from the first to the second line, we have got rid of the singular term $\eps^{-1}U^\perp$ in the first integral. Indeed, the condition $\div U=0$ implies that $U^\perp$ is a
perfect gradient; since $\div \de u_\eps=0$, in turn we get that the term under consideration actually vanishes.

Thus, let us come back to \eqref{eq:I_1}, and
focus on $I_{11}$ for a while. For dealing with this term, we can use the second equation in \eqref{EQLim} to write
\begin{align*}
I_{11}\,=\,-\int_\Omega\left( (\de u_\eps\cdot \nabla) U - RU^\perp + \nu(1) \Delta U + (B\cdot\nabla) B\right)\cdot\de u_\eps\,,
\end{align*}
where, once again, we have omitted to write the terms which are pure gradients, since $\div\de u_\eps=0$.
Combining this term with $I_5$ and $I_6$, we find, after integration by parts
\begin{align*}
I_{11}\,+\,I_5\,+\,I_6\,&=\,-\int_\Omega\left( (\de u_\eps\cdot\nabla) U - (\de b_\eps\cdot\nabla) B - RU^\perp\right)\cdot\de u_\eps\,-\,\int_\Omega\big(\nu(\rho_\eps)-\nu(1)\big)\nabla U:\nabla\de u_\eps\,.
\end{align*}
Plugging the expression of $I_{12}$ into this equation, we finally get
\begin{align}
I_1\,+\,I_5\,+\,I_6\,&=\,-\int_\Omega\left( (\de u_\eps\cdot\nabla) U - (\de b_\eps\cdot\nabla) B - \de r_\eps U^\perp\right)\cdot\de u_\eps \label{eq:U-term} \\
&\qquad\qquad-\int_\Omega\big(\nu(\rho_\eps)-\nu(1)\big)\nabla U:\nabla\de u_\eps-\,\eps\int_\Omega r_\eps\big(\d_tU+ (u_\eps\cdot\nabla) U\big)\cdot\de u_\eps\,. \nonumber
\end{align}

In the end, by use of \eqref{eq:R-term}, \eqref{eq:B-term} and \eqref{eq:U-term}, we deduce that
\begin{align}
\mc R_\eps\,=&\,-\int_\Omega\de r_\eps\,\de u_\eps\cdot\nabla R\,+\,  \int_\Omega \Big( (\de b_\eps\cdot\nabla) U  -  (\du \cdot \nabla)B \Big)\cdot\de b_\eps \label{eq:RemFinal} \\
&\,-\,\int_\Omega\big(\mu(\rho_\eps)-\mu(1)\big)\big(\nabla\times B\big)\,\big(\nabla\times \de b_\eps\big) \nonumber \\
&\,-\,\int_\Omega\left( (\de u_\eps\cdot\nabla) U - (\de b_\eps\cdot\nabla) B - \de r_\eps U^\perp\right)\cdot\de u_\eps\, \nonumber \\
&\,-\,\int_\Omega\big(\nu(\rho_\eps)-\nu(1)\big)\nabla U:\nabla\de u_\eps\,-\,\eps\int_\Omega r_\eps\big(\d_tU+ (u_\eps\cdot\nabla) U\big)\cdot\de u_\eps\,=\,\sum_{\ell=1}^6J_\ell\,. \nonumber 
\end{align}

Our next goal is to bound all the terms $J_\ell$ appearing in the previous expression: this will be done in the next paragraph. Notice that
those bounds will complete the proof to Theorem \ref{THSt}.
Indeed, since the density is a perturbation of a constant state, \tsl{i.e.} $\rho_\veps = 1 + \veps r_\veps$, the relative entropy $\mathcal{E}\Big([r_\eps,u_\eps,b_\eps]\,\Big|\,[R,U,B]\Big)$
is in fact equivalent to the $L^2$ norm of the error function $(\dr, \du, \db)$.

\subsubsection{Quantitative estimates for the viscous resistive system} \label{sss:quantitative}

In the computations below, we make extensive use of the Gagliardo-Nirenberg inequality (GN for short) reproduced in Proposition \ref{GN}.
According to the notation introduced above, we agree to note $M_p(t) \in L^p(\R_+)$ generic globally $L^p$ functions; on the other hand, we will use the notation $N_p(t)$ to denote functions
in $L^p_{\rm loc}(\R_+)$.

\medbreak
Let us bound all the terms appearing in \eqref{eq:RemFinal}. 
We start by handling $J_1$. Recall that $R \in L^\infty_T(H^{1 + \gamma})$ for any $0 \leq \gamma < \beta$. Making use of H\"older's and GN inequalities, we get
\begin{align*}
\left|  \int_\Omega (\nabla R \cdot \du) \dr \dx  \right| & \leq \| \dr \|_{L^2} \| \du \|_{L^p} \| \nabla R \|_{L^q} \leq
C \| \nabla R \|_{L^q} \| \dr \|_{L^2} \| \du \|_{L^2}^{2/p} \| \nabla( \du) \|_{L^2}^{1 - 2/p}\\
& \leq  \eta \| \nabla (\du) \|_{L^2}^2 + C(\eta,q) \| \nabla R \|_{L^q}^{q'} \| \dr \|_{L^2}^{q'} \| \du \|_{L^2}^{2q'/p}\,,
\end{align*}
where $\eta > 0$ is arbitrarily small, $p, q \geq 2$ are chosen so that $\frac{1}{p} + \frac{1}{q} = \frac{1}{2}$ and the exponent $q'$ is associated to $q$ in Young's inequality by
$\frac{1}{q} + \frac{1}{q'} = 1$. Using Young's inequality one more time with the exponents $\alpha = \frac{2(q-1)}{q}$ and $\beta = \frac{2(q-1)}{q-2}$ (which satisfy
$\frac{1}{\alpha} + \frac{1}{\beta} = 1$) allows us to introduce the relative entropy function $\mc E (t)$ in the right-hand side:
\begin{equation*}
\left|  \int_\Omega (\nabla R \cdot \du) \dr \dx  \right| \leq   \eta \| \nabla(\du) \|_{L^2}^2 + C(\eta,q) \left(1+\| \nabla R \|^{2}_{L^q}\right) \mc E(t)\,.
\end{equation*}
Now since $\nabla R \in L^\infty_T(H^\gamma)$, we see that $\nabla R \in L^\infty_T(L^q)$ for $q$ close enough to $2$, thanks to Sobolev embedding. For such $q$, it is always possible to find a
$p \geq 2$ with $\frac{1}{p} + \frac{1}{q} = \frac{1}{2}$, so that all of the preceding inequalities are justified. \textsl{In fine}, using Proposition \ref{p:LimitEstimates},
we find the following inequality:
\begin{equation} \label{est:J_1}
\left|  \int_\Omega (\nabla R \cdot \du) \dr \dx  \right| \leq  \eta \| \nabla (\du) \|_{L^2}^2 + C(\eta,q,\|r_0\|_{H^{1+\beta}},T)  \mc E(t)\,.
\end{equation}

Next, we look at $J_2$. Using H\"older's inequality with exponents $\frac{1}{2} + \frac{1}{4} + \frac{1}{4} = 1$, we get 
\begin{equation*}
\left|  \int_\Omega (\db \cdot \nabla) U \cdot \db \dx  \right| \, \leq\, \| \nabla U \|_{L^2}  \| \db \|_{L^4}^2 \,\leq\, \| \nabla U \|_{L^2} \| \db \|_{L^2} \| \nabla \db \|_{L^2},
\end{equation*}
where we have also exploited GN inequality. Using now Young's inequality, we infer, for all $\eta > 0$, the bound
\begin{equation*}
\begin{split}
\left|  \int_\Omega (\db \cdot \nabla) U \cdot \db \dx  \right| & \leq \eta \| \nabla \db \|_{L^2}^2 + C(\eta) \| \nabla U \|_{L^2}^2 \| \db \|_{L^2}^2\\
& = \eta \| \nabla \db \|_{L^2}^2 + M_1(t) \| \db \|_{L^2}^2,
 \end{split}
\end{equation*}
where we have set $M_1(t)\,=\,C(\eta)\| \nabla U(t) \|_{L^2}^2$. Notice that, thanks to the estimates of Proposition \ref{p:LimitEstimates}, one has
$M_1\in L^1(\R_+)$, with $\| M_1 \|_{L^1(\mathbb{R}_+)} = C(\eta, \| u_0 \|_{L^2}, \| b_0 \|_{L^2})$.
In the very same way, we also have 
\begin{equation*}
\begin{split}
\left| \int_\Omega (\du \cdot \nabla) B \cdot \db \dx \right| & \leq  \| \nabla B \|_{L^2} \| \du \|_{L^4} \| \db \|_{L^4}\\
& \leq C \| \nabla B \|_{L^2}  \| \du \|_{L^2}^{1/2} \| \nabla \du \|_{L^2}^{1/2} \|  \db \|_{L^2}^{1/2}  \| \nabla \db \|_{L^2}^{1/2}\\
& \leq \eta \Big( \| \nabla \du \|_{L^2}^2 + \| \nabla \db \|_{L^2}^2 \Big) + C(\eta) \| \nabla B \|_{L^2}^2 \Big( \| \du \|_{L^2}^2 + \| \db \|_{L^2}^2 \Big)\\
& = \eta \Big( \| \nabla \du \|_{L^2}^2 + \| \nabla \db \|_{L^2}^2 \Big) + M_1(t) \mc E(t),
\end{split}
\end{equation*}
where $M_1(t) = \| \nabla B(t) \|_{L^2}$ is, as above, an $L^1(\R_+)$ function whose $L^1$ norm is bounded by a constant $C = C(\eta, \| u_0 \|_{L^2}, \| b_0 \|_{L^2})$.
In the end, we have proved the following bound for $J_2$:
\begin{equation} \label{est:J_2}
 \left|J_2\right|\,\leq\,2\eta \Big( \| \nabla \du \|_{L^2}^2 + \| \nabla \db \|_{L^2}^2 \Big) + M_1(t) \mc E(t)\,,
\end{equation}
where $\eta>0$ can be chosen arbitrarily small.

We now consider $J_3$. By assumption, $\mu$ is $\sigma$-continuous in a neighbourhood of $1$, with $\s$ being non-decreasing. Therefore, for $\veps>0$ so small that
$\veps M\leq 1$, with $M$ defined in the statement of Theorem \ref{THSt}, we can estimate
\begin{align} \label{est:J_3}
\left|  \int_\Omega \big( \mu(\rho_\veps) - \mu(1) \big) (\nabla \times  B) \cdot ( \nabla \times \db)  \right| & \leq
| \mu |_{C_\sigma} \sigma\left(\veps \| r_\veps \|_{L^\infty}\right) \| \nabla \times B \|_{L^2} \| \nabla \times \db \|_{L^2}\\
& \leq \eta \| \nabla \times \db \|_{L^2}^2 + C(\eta, |\mu|_{C_\s}) \sigma(M\veps)^2 \| \nabla \times B \|_{L^2}^2 \nonumber \\
& = \eta \| \nabla \times \db \|_{L^2}^2 + \sigma(M \veps)^2 M_1(t)\,, \nonumber
\end{align}
where $M_1\in L^1(\R_+)$, with (in view of Proposition \ref{p:LimitEstimates}) $\| M_1 \|_{L^1(\mathbb{R}_+)}$ depending only on $\eta$, $| \mu|_{C_\s}$, $\| u_{0} \|_{L^2}$ and $\| b_{0} \|_{L^2}$.
The integral $J_5$ containing the viscosity term is dealt with in the same way.

For $J_4$, we separate the integral into three summands:
\begin{equation*}
J_4 = \int_\Omega (\db \cdot \nabla) B \cdot \du \dx - \int_\Omega (\du \cdot \nabla) U \cdot \du \dx -  \int_\Omega \dr U^\perp \cdot \du \dx := J_{4,1} + J_{4,2} + J_{4,3}\,.
\end{equation*}
The first term $J_{4,1}$ can be dealt with in a way analogous to the second term appearing in $J_2$: combining H\"older and GN inequalities, we get
\[ 
\left|  \int_\Omega (\db \cdot \nabla) B \cdot \du \dx \right|\,\leq\, \| \du \|_{L^4} \| \db \|_{L^4} \| \nabla B \|_{L^2}\,
\leq\, \eta \Big( \| \nabla \du \|_{L^2}^2 + \| \nabla \db \|_{L^2}^2 \Big) + M_1(t) \mc E(t)\,,
\]
where $M_1 \in L^1(\R_+)$, with $\| M_1 \|_{L^1(\mathbb{R}_+)}$ being a function of $(\eta, \| u_0 \|_{L^2}, \| b_0 \|_{L^2})$.
Now look at $J_{4,2}$: 
a very similar argument yields, for $M_1(t)\,=\,C(\eta) \| \nabla U (t) \|_{L^2}^2$, with $M_1\in L^1(\R_+)$, the inequality
\[ 
\left|  \int_\Omega (\du \cdot \nabla) U \cdot \du \dx  \right| \leq \eta \| \nabla(\du) \|_{L^2}^2 +  M_1(t) \mc E(t)\,.
\] 
Finally, we notice that $J_{4,3}$ is very similar to $J_1$, up to substituting $\nabla R$ with $U$.
In fact, since $U \in L^\infty_T(H^1)$ (see Proposition \ref{p:LimitEstimates}), and not only $L^\infty_T(H^\gamma)$ as $\nabla R$ before, it suffices to conduct the computations for any values
of $p$ and $q$: taking for simplicity $p = q = 4$, we deduce
\begin{align*}
\left|  \int_\Omega \dr U^\perp \cdot \du \dx  \right| & \leq \| U \|_{L^4} \| \dr \|_{L^2} \| \du \|_{L^4} \leq
\eta \| \nabla(\du) \|^2_{L^2} + C(\eta) \| U \|^{4/3}_{L^\infty_T(H^1)}   \| \dr \|_{L^2}^{4/3} \| \du \|_{L^2}^{2/3} \\
& \leq \eta \| \nabla (\du) \|_{L^2}^2 + C(\eta,\|u_0\|_{H^1}, \| b_0 \|_{H^1}, T)\,  \mc E(t)\,.
\end{align*}
Summing up all the last inequalities, we gather
\begin{equation} \label{est:J_4}
 \left|J_4\right|\,\leq\,3\eta \Big( \| \nabla \du \|_{L^2}^2 + \| \nabla \db \|_{L^2}^2 \Big) + \big(M_1(t)\,+\,C\big) \mc E(t)\,,
\end{equation}
where, for simplicity, we have omitted the various quantities on which the constant $C>0$ depends.

It remains to bound $J_6$. On the one hand, the integral containing the time derivative can be bounded by using Proposition \ref{p:LimitEstimates}:
\begin{align*}
\veps \left| \int_\Omega r_\veps \d_t U \cdot \du \dx \right| & \leq \veps \| \partial_t U \|_{L^2} \| r_\veps \|_{L^\infty} \| \du \|_{L^2}\,\leq\,
\veps^2 C(\| r_{0, \veps} \|_{L^\infty}) \| \partial_t U \|_{L^2}^2 + \mc E(t)\\
& \leq \veps^2 N_1(t) + \mc E(t),
\end{align*}
where $\| N_1 \|_{L^1_T}$ grows at polynomial speed $1+T^h$ and depends on $(h,\| r_{0, \veps} \|_{L^\infty}, \| u_0 \|_{H^1}, \| b_0 \|_{H^1}, T)$.
On the other hand, using H\"older's and GN inequalities we infer
\begin{align*}
\left|  \veps \int_\Omega r_\veps (u_\veps \cdot \nabla) U \cdot \du \dx  \right| & \leq \veps \| r_\veps \|_{L^\infty} \| u_\veps \|_{L^4} \| \nabla U \|_{L^2} \| \du \|_{L^4}\\
& \leq \veps C(\| r_{0, \veps} \|_{L^\infty}) \| u_\veps \|_{L^2}^{1/2} \| \nabla u_\veps \|_{L^2}^{1/2} \| \nabla U \|_{L^2} \| \du \|_{L^2}^{1/2} \| \nabla (\du) \|_{L^2}^{1/2}
\end{align*}
Recall that $u_\veps \in L^\infty (L^2)$. Using Young's inequality a first time with coefficients $\frac{1}{4} + \frac{3}{4} = 1$ yields
\[ 
\left|  \veps \int_\Omega r_\veps (u_\veps \cdot \nabla) U \cdot \du \dx  \right| \leq \eta \| \nabla (\du) \|_{L^2}^2
+ \veps^{4/3} C \| \nabla u_\veps \|_{L^2}^{2/3} \| \nabla U \|^{4/3}_{L^2} \| \du \|_{L^2}^{2/3}\,,
\]
with $C=C(\eta, \| r_{0, \veps} \|_{L^\infty}, \| u_{0, \veps} \|_{L^2}, \| b_{0, \veps} \|_{L^2})$, and a second time on the second summand with coefficients $\frac{1}{3} + \frac{2}{3} = 1$,
we gather
\begin{align*}
\left|  \veps \int_\Omega r_\veps (u_\veps \cdot \nabla) U \cdot \du \dx  \right| &\leq \eta \| \nabla (\du) \|_{L^2}^2 +  \| \du \|_{L^2}^2 \| \nabla u_\veps \|_{L^2}^2
+ \veps^{2} C \| \nabla U \|_{L^2}^2 \\
&\leq \eta \| \nabla (\du) \|_{L^2}^2 + M_1(t) \mc E(t) + \veps^{2} M_1(t)\,,
\end{align*}
where $C>0$ depends on the same quantities as the previous constant, and we have used the fact that both $\nabla u_\veps$ and $\nabla U$ belong to $L^2(\R_+;L^2)$ to introduce
the function $M_1\in L^1(\R_+)$. Notice that the $L^1$ norm of $M_1$
depends on $(\eta, \| r_{0, \veps} \|_{L^\infty}, \| u_{0, \veps} \|_{L^2}, \| b_{0, \veps} \|_{L^2})$.
In the end, we deduce that
\begin{equation} \label{est:J_6}
\left| J_6 \right|\,\leq\,\eta \Big( \| \nabla \du \|_{L^2}^2 + \| \nabla \db \|_{L^2}^2 \Big) + \big(M_1(t)\,+\,1\big) \mc E(t)\,+\,\veps^2\,\big(M_1(t)\,+\,N_1(t)\big)\,.
\end{equation}

Piecing inequalities \eqref{est:J_1}, \eqref{est:J_2}, \eqref{est:J_3}, \eqref{est:J_4}, \eqref{est:J_6} all together and taking $\eta$ small enough, say $\eta = \frac{1}{100} \min \{ \nu_*, \mu_* \}$,
we find
\begin{align*}
&\mc E(T) + \int_0^T \int_\Omega \bigg\{ \nu_* | \nabla \du |^2 + \mu_* | \nabla \db |^2  \bigg\} \dx \dt \\
&\qquad
\leq C
\left(\mc E\Big([r_{0,\eps},u_{0,\eps},b_{0,\eps}]\,\Big|\,[R_0,U_0,B_0]\Big) + \int_0^T \Big( \mathcal{M}_1(t) \mc E(t) + \max\left\{\veps^2\,,\,\s^2(M \veps)\right\} \mathcal{N}_1(t) \Big)  {\rm d}t\right)
\end{align*}
for any $T>0$,
with $\mathcal{M}_1$ and $\mathcal{N}_1$ being locally integrable functions on $\mathbb{R}_+$. 
Use of Gr\"onwall's lemma on this differential inequality provides the result we covet, namely inequality \eqref{EQRelIn}.

The proof of Theorem \ref{THSt} is thus completed.

\subsection{Vanishing viscosity and resistivity limit: derivation of the ideal system} \label{ss:VanVisc}

In this subsection, we show the proof of Theorem \ref{th:InvConv}, concerning the derivation of the ideal system, which corresponds to the case $h(\veps) \rightarrow 0^+$. With respect to the previous case, we lose any control on the gradients
of the quantities $\de u_\veps$ and $\de b_\veps$, since we have to deal with a vanishing viscosity and resistivity limit. On the other hand, the solution $(R,U,B)$ to the limit problem
will enjoy, on its lifespan, much more smoothness than in the previous section. In addition, we point out that the convergence here is limited to the time $T^*$ representing the lifespan
of $(R,U,B)$, which is possibly finite.

%

%
Also in this section, the main ingredient is the relative entropy inequality of Theorem \ref{t:rel-en}. We skip the proof of the fact that $(R,U,B)$ verifies indeed the regularity hypotheses
of that statement.

So, let us write the relative entropy inequality \eqref{est:rel-entropy} for $(r, u, b)$ and $(R, U, B)$: we get
\begin{multline}\label{eq:entrId}
\mc E \Big( [r_\veps, u_\veps, b_\veps] \big| [R, U, B] \Big)(T) + h(\veps) \int_0^T \int_\Omega \bigg\{ \nu(\rho_\veps) |\nabla \du|^2 + \mu(\rho_\veps)|\nabla\times \db|^2 \bigg\} \dx \dt \\
\leq \mc E\Big([r_{0,\eps},u_{0,\eps},b_{0,\eps}]\,\Big|\,[R_0,U_0,B_0]\Big) \int_0^T \mc R_\veps \dt.
\end{multline}
Performing exactly the same computations as in Paragraph \ref{sss:rel-entr-rot}, we get an expression for the reminder term analogous to \eqref{eq:RemFinal}: 
\begin{align*}
\mc R_\eps\,=&\,-\int_\Omega\de r_\eps\,\de u_\eps\cdot\nabla R\,+\,  \int_\Omega \Big( (\de b_\eps\cdot\nabla) U  -  (\du \cdot \nabla)B \Big)\cdot\de b_\eps \\ 
&\,-\,h(\veps) \int_\Omega \mu(\rho_\eps) \big(\nabla\times B\big)\,\big(\nabla\times \de b_\eps\big)\,-\,
\int_\Omega\left( (\de u_\eps\cdot\nabla) U - (\de b_\eps\cdot\nabla) B - \de r_\eps U^\perp\right)\cdot\de u_\eps\, \nonumber \\
&\,-\, h(\veps) \int_\Omega \nu(\rho_\eps) \nabla U:\nabla\de u_\eps\,-\,\eps\int_\Omega r_\eps\big(\d_tU+ (u_\eps\cdot\nabla) U\big)\cdot\de u_\eps\,=\,\sum_{\ell=1}^6J_\ell\,. \nonumber 
\end{align*}
We are going to bound all the integrals $J_1,\ldots, J_6$ one after the other.

First of all, for $J_1$ we have
\begin{equation*}
\begin{split}
\left| \int_\Omega \dr \du \cdot \nabla R \dx \right| & \leq \| \nabla R \|_{L^\infty} \| \dr \|_{L^2} \| \du \|_{L^2} \leq \| R \|_{L^\infty_T (H^s)} \Big( \| \dr \|_{L^2}^2 + \| \du \|_{L^2}^2 \Big) \\
&\leq C\big( T, \| (R_0, U_0, B_0) \|_{H^s} \big) \mc E(t)\,.
\end{split}
\end{equation*}
As for $J_2$, we argue in the very same way and use Sobolev inequality to get
\begin{equation*}
\begin{split}
\left| \int_\Omega \Big( (\db \cdot \nabla) U \cdot \db - (\du \cdot \nabla)B \cdot \db \Big) \dx \right| & \leq \Big( \|\nabla U \|_{L^\infty} + \|\nabla B \|_{L^\infty} \Big) \Big( \| \du \|_{L^2}^2 + \| \db \|_{L^2}^2 \Big) \\
&\leq C\big( T, \| (R_0, U_0, B_0) \|_{H^s} \big) \mc E(t)\,.
\end{split}
\end{equation*}
The fourth integral $J_4$ is dealt with in the same manner:
\begin{equation*}
\begin{split}
\bigg| \int_\Omega\Big( (\de u_\eps\cdot\nabla) U & + (\de b_\eps\cdot\nabla) B - \de r_\eps U^\perp\Big)\cdot\de u_\eps \dx \bigg| \\
& \leq  \| \nabla U \|_{L^\infty} \| \du \|_{L^2}^2 + \| \nabla B \|_{L^\infty} \| \du \|_{L^2} \| \db \|_{L^2} +  \| U \|_{L^\infty} \| \dr \|_{L^2} \| \du \|_{L^2}\\
& \leq C \big( T, \| (R_0, U_0, B_0) \|_{H^s} \big)\,\mc E(t)\,.
\end{split}
\end{equation*}

Now we take care of the integrals containing the derivatives of the error functions $\du$ and $\db$, namely $J_3$ and $J_5$.
Here, we use the fact that $\nabla \du$ and $\nabla \db$ have $L^2$ regularity, even though this property is not uniform with respect to $\veps$. More precisely,
the energy inequality for the primitive system \eqref{MHD1}, see item (vii) of Definition \ref{d:weak}, yields, for any $T>0$, the uniform bound
$$
\big\| \sqrt{h(\veps)} \,\, \nabla u_\veps \big\|_{L^2_T(L^2)} + \big\| \sqrt{h(\veps)} \,\, \mu(\rho_\veps) \nabla \times b_\veps \big\|_{L^2_T(L^2)} \leq C \big( \| u_{0, \veps} \|_{L^2}, \| b_{0, \veps}\|_{L^2} \big).
$$
This means that the derivatives of the difference functions $\nabla \du$ and $\nabla \db$ also have $L^2$ regularity and, thanks to the entropy inequality \eqref{eq:entrId}, 
they will enjoy similar bounds. Thus, for any small $\eta > 0$, we can estimate
\begin{multline*}
h(\veps) \left| \int_\Omega \mu(\rho_\veps) (\nabla \times B) (\nabla \times \db) \dx \right| \leq \sqrt{h(\veps)}\,\,  \| \nabla \times B \|_{L^2} \big\| \sqrt{h(\veps)}\,\, \mu(\rho_\veps) \nabla \times \db \big\|_{L^2} \\
\leq \eta \big\| \sqrt{h(\veps)} \,\, \mu(\rho_\veps) \nabla \times \db \big\|_{L^2_T(L^2)}^2 + h(\veps)\, C \big( T, \eta, \| (R_0, U_0, B_0) \|_{H^s} \big).
\end{multline*}
Exactly in the same way, we also have 
\begin{equation*}
h(\veps) \left| \int_\Omega \nu(\rho_\veps) \nabla  U :  \nabla \du \dx \right| \leq \eta \big\| \sqrt{h(\veps)} \,\, \nu(\rho_\veps) \nabla \du \big\|_{L^2_T(L^2)}^2 + h(\veps)\,
C \big( T, \eta, \| (R_0, U_0, B_0) \|_{H^s} \big).
\end{equation*}

Only the last integral $J_6$ remains. It involves the time derivative $\partial_t U$, whose $L^\infty_T(H^{s-1})$ regularity (for all $T<T^*$) is given by Theorem \ref{th:BesovWP}
above. Using the embedding $H^{s-1} \hookrightarrow L^\infty$, we finally gather
\begin{equation*}
\begin{split}
\veps \bigg| \int_\Omega r_\veps \big( \partial_t U + (u_\veps \cdot \nabla) U \big) \cdot \du \dx  \bigg| & \leq \veps \| \partial_t U \|_{L^\infty} \| r_\veps \|_{L^2} \| \du \|_{L^2} + \veps \| u_\veps \|_{L^2} \| \nabla U \|_{L^\infty} \| \du \|_{L^2}\\
& \leq \| \du \|_{L^2}^2 + \veps^2 \| \partial_t U \|_{H^{s-1}}^2 \| r_\veps \|^2_{L^2} + \veps^2 \| U \|_{H^s}^2 \| u_\veps \|_{L^2}^2   \\
& \leq \| \du \|_{L^2}^2 + \veps^2 C \big( T, \| (R_0, U_0, B_0) \|_{H^s},\| (r_{0, \veps},u_{0, \veps}, b_{0, \veps}) \|_{L^2} \big).
\end{split}
\end{equation*}

Putting all the estimates for $J_1, ..., J_6$ together and choosing $\eta$ small enough, we get
\begin{multline*}
\mc E \Big( [r_\veps, u_\veps, b_\veps] \big| [R, U, B] \Big)(T) + \frac{1}{2} h(\veps) \int_0^T \int_\Omega \bigg\{ \nu(\rho_\veps) |\nabla \du|^2 + \mu(\rho_\veps)|\nabla \db|^2 \bigg\} \dx \dt \\
\leq \int_0^T \mc E \Big( [r_\veps, u_\veps, b_\veps] \big| [R, U, B] \Big) \dt + \left(h(\veps) + \veps^2 \right) \, C,
\end{multline*}
for a suitable constant $C\,=\,C\Big( T, \| (R_0, U_0, B_0) \|_{H^s}, \| (r_{0, \veps},u_{0, \veps}, b_{0, \veps}) \|_{L^2} \Big)>0$. An application of Gr\"onwall's lemma gives estimate \eqref{EQRelIn2},
completing in this way the proof to Theorem \ref{th:InvConv}.


\appendix

\section{Appendix -- Fourier and harmonic analysis toolbox} \label{app:LP}

In this appendix, we collect tools from Fourier analysis and Littlewood-Paley theory which we have freely used throughout all our paper. We refer to Chapters 2 and 3 of \cite{BCD} for details.

\subsection{Non-homogeneous Littlewood-Paley theory and Besov spaces}

We recall here the main ideas of Littlewood-Paley theory, which we exploited in the previous analysis.
For simplicity of exposition, let us deal with the $\R^d$ case; however, the whole construction can also be adapted to the $d$-dimensional torus $\T^d$.

First of all, let us introduce the so-called ``Littlewood-Paley decomposition''.
We fix a smooth radial function $\chi$ supported in the ball $B(0,2)$, equal to $1$ in a neighbourhood of $B(0,1)$
and such that $r\mapsto\chi(r\,e)$ is non-increasing over $\R_+$ for all unitary vectors $e\in\R^d$. Set
$\varphi\left(\xi\right)=\chi\left(\xi\right)-\chi\left(2\xi\right)$ and
$\vphi_j(\xi):=\vphi(2^{-j}\xi)$ for all $j\geq0$.
The dyadic blocks $(\Delta_j)_{j\in\Z}$ are defined by\footnote{Throughout we agree  that  $f(D)$ stands for 
the pseudo-differential operator $u\mapsto\mc{F}^{-1}[f(\xi)\,\what u(\xi)]$.} 
$$
\Delta_j\,:=\,0\quad\mbox{ if }\; j\leq-2,\qquad\Delta_{-1}\,:=\,\chi(D)\qquad\mbox{ and }\qquad
\Delta_j\,:=\,\varphi(2^{-j}D)\quad \mbox{ if }\;  j\geq0\,.
$$
We  also introduce the following low frequency cut-off operator:
\begin{equation*} \label{eq:S_j}
S_ju\,:=\,\chi(2^{-j}D)\,=\,\sum_{k\leq j-1}\Delta_{k}\qquad\mbox{ for }\qquad j\geq0\,.
\end{equation*}
Note that the operator $S_j$ is a convolution operator with a function $K_j(x) = 2^{dj}K_1(2^j x) = \mathcal{F}^{-1}[\chi (2^{-j} \xi)] (x)$ of constant $L^1$ norm,
and hence defines a continuous operator for the $L^p \longrightarrow L^p$ topologies, for any $1 \leq p \leq +\infty$.

\medskip

The following classical property holds true: for any $u\in\mc{S}'$, then one has the equality~$u=\sum_{j}\Delta_ju$ in the sense of $\mc{S}'$.
Let us also mention the so-called \emph{Bernstein inequalities}.
  \begin{lemma} \label{l:bern}
Let  $0<r<R$.   A constant $C$ exists so that, for any nonnegative integer $k$, any couple $(p,q)$ 
in $[1,+\infty]^2$, with  $p\leq q$,  and any function $u\in L^p$,  we  have, for all $\lambda>0$,
$$
\displaylines{
{\Supp}\, \widehat u \subset   B(0,\lambda R)\quad
\Longrightarrow\quad
\|\nabla^k u\|_{L^q}\, \leq\,
 C^{k+1}\,\lambda^{k+d\left(\frac{1}{p}-\frac{1}{q}\right)}\,\|u\|_{L^p}\;;\cr
{\Supp}\, \widehat u \subset \{\xi\in\R^d\,|\, r\lambda\leq|\xi|\leq R\lambda\}
\quad\Longrightarrow\quad C^{-k-1}\,\lambda^k\|u\|_{L^p}\,
\leq\,
\|\nabla^k u\|_{L^p}\,
\leq\,
C^{k+1} \, \lambda^k\|u\|_{L^p}\,.
}$$
\end{lemma}   

By use of Littlewood-Paley decomposition, we can define the class of Besov spaces.
\begin{defi} \label{d:B}
  Let $s\in\R$ and $1\leq p,r\leq+\infty$. The \emph{non-homogeneous Besov space}
$B^{s}_{p,r}$ is defined as the subset of tempered distributions $u$ for which
$$
\|u\|_{B^{s}_{p,r}}\,:=\,
\left\|\left(2^{js}\,\|\Delta_ju\|_{L^p}\right)_{j\geq -1}\right\|_{\ell^r}\,<\,+\infty\,.
$$
\end{defi}
Besov spaces are interpolation spaces between Sobolev spaces. In fact, for any $k\in\N$ and~$p\in[1,+\infty]$
we have the following chain of continuous embeddings: $ B^k_{p,1}\hookrightarrow W^{k,p}\hookrightarrow B^k_{p,\infty}$,
where  $W^{k,p}$ denotes the classical Sobolev space of $L^p$ functions with all the derivatives up to the order $k$ in $L^p$.
When $1<p<+\infty$, we can refine the previous result (this is the non-homogeneous version of Theorems 2.40 and 2.41 in \cite{BCD}): we have
$$
 B^k_{p, \min (p, 2)}\hookrightarrow W^{k,p}\hookrightarrow B^k_{p, \max(p, 2)}\,.
$$
In particular, for all $s\in\R$ we deduce the equivalence $B^s_{2,2}\equiv H^s$, with equivalence of norms.

As an immediate consequence of the first Bernstein inequality, one gets the following embedding result.
\begin{prop}\label{p:embed}
The continuous embedding $B^{s_1}_{p_1,r_1}\,\hookrightarrow\,B^{s_2}_{p_2,r_2}$ holds whenever $p_1\,\leq\,p_2$ and
$$
s_2\,<\,s_1-d\left(\frac{1}{p_1}-\frac{1}{p_2}\right)\,,\qquad\mbox{ or }\qquad
s_2\,=\,s_1-d\left(\frac{1}{p_1}-\frac{1}{p_2}\right)\;\;\mbox{ and }\;\;r_1\,\leq\,r_2\,. 
$$
\end{prop}


In particular, under the conditions 
\begin{equation}\label{eq:AnnLInfty}
s > \frac{d}{p}\,, \qquad\qquad \text{ or } \qquad\qquad s = \frac{d}{p} \quad \text{ and } \quad r = 1,
\end{equation}
on $(s, p, r) \in \mathbb{R} \times [1, +\infty]^2$, we get the following chain of embeddings:
\begin{equation*}
\B \hookrightarrow B^{s - \frac{d}{p}}_{p, r} \hookrightarrow B^0_{\infty, 1} \hookrightarrow L^\infty.
\end{equation*}

Now, let us make an important remark about the Leray projector $\P$.
\begin{rmk}\label{r:Leray}
The Leray projector $\mathbb{P} = I + \nabla (- \Delta)^{-1} \D$ is a Fourier multiplier whose symbol is a bounded rational fraction of order zero. By use of Calder\'on-Zygmund theory, it can be shown that $\P$ maps \emph{continuously} $\B$ into itself, as long as $1 < p < +\infty$ and for all $(s, r) \in \mathbb{R} \times [1, +\infty]$. 
%
\end{rmk}

The next proposition is a functional inequality which we used repeatedly in this article, the classical \emph{Gagliardo-Nirenberg inequality}. Its proof can be found \tsl{e.g.} in
Corollary 1.2 of \cite{CDGG}.

\begin{lemma}\label{GN}
Let $2 \leq p < +\infty$ such that $1/p > 1/2 - 1/d$. Then, for all $u \in H^1$, one has
\begin{equation*}
\| u \|_{L^p}\, \leq\, C(p)\, \| u \|_{L^2}^{1-\lambda}\, \| \nabla u \|_{L^2}^{\lambda}\,, \qquad\qquad \text{ with } \qquad \lambda = \frac{d(p-2)}{2p}\,.
\end{equation*}
In particular, in dimension $d = 2$, we have $\| u \|_{L^p}\, \leq\, \| u \|_{L^2}^{2/p}\, \| \nabla u \|_{L^2}^{1 - 2/p}$ for any $p\in[2,+\infty[\,$.
\end{lemma}

\subsection{Non-homogeneous paradifferential calculus}\label{s:NHPC}

Let us now introduce the paraproduct operator (after J.-M. Bony, see \cite{Bony}). Constructing the paraproduct operator relies on the observation that, 
formally, any product  of two tempered distributions $u$ and $v,$ may be decomposed into 
\begin{equation}\label{eq:bony}
u\,v\;=\;\mathcal{T}_u(v)\,+\,\mathcal{T}_v(u)\,+\,\mathcal{R}(u,v)\,,
\end{equation}
where we have defined
$$
\mathcal{T}_u(v)\,:=\,\sum_jS_{j-1}u\Delta_j v,\qquad\qquad\mbox{ and }\qquad\qquad
\mathcal{R}(u,v)\,:=\,\sum_j\sum_{|j'-j|\leq1}\Delta_j u\,\Delta_{j'}v\,.
$$
The above operator $\mc T$ is called ``paraproduct'' whereas
$\mc R$ is called ``remainder''.
The paraproduct and remainder operators have many nice continuity properties. 
The following ones have been of constant use in this paper. 
\begin{prop}\label{p:op}
For any $(s,p,r)\in\R\times[1,+\infty]^2$ and $t>0$, the paraproduct operator 
$\mathcal{T}$ maps continuously $L^\infty\times B^s_{p,r}$ in $B^s_{p,r}$ and  $B^{-t}_{\infty,\infty}\times B^s_{p,r}$ in $B^{s-t}_{p,r}$.
Moreover, the following estimates hold:
$$
\|\mathcal{T}_u(v)\|_{B^s_{p,r}}\,\leq\, C\,\|u\|_{L^\infty}\,\|\nabla v\|_{B^{s-1}_{p,r}}\qquad\mbox{ and }\qquad
\|\mathcal{T}_u(v)\|_{B^{s-t}_{p,r}}\,\leq\, C\|u\|_{B^{-t}_{\infty,\infty}}\,\|\nabla v\|_{B^{s-1}_{p,r}}\,.
$$
For any $(s_1,p_1,r_1)$ and $(s_2,p_2,r_2)$ in $\R\times[1,+\infty]^2$ such that 
$s_1+s_2>0$, $1/p:=1/p_1+1/p_2\leq1$ and~$1/r:=1/r_1+1/r_2\leq1$,
the remainder operator $\mathcal{R}$ maps continuously~$B^{s_1}_{p_1,r_1}\times B^{s_2}_{p_2,r_2}$ into~$B^{s_1+s_2}_{p,r}$.
In the case $s_1+s_2=0$, provided $r=1$, operator $\mathcal{R}$ is continuous from $B^{s_1}_{p_1,r_1}\times B^{s_2}_{p_2,r_2}$ with values
in $B^{0}_{p,\infty}$.
\end{prop}

The consequence of this proposition is that the spaces $\B$ are Banach algebras as long as the condition \eqref{eq:AnnLInfty} holds with $s > 0$. Moreover, in that case, we have the so-called \textit{tame estimates}.

\begin{cor}\label{c:tame}
Let $(s, p, r)$ be as in \eqref{eq:AnnLInfty} with the extra assumption that $s > 0$. Then, we have
\begin{equation*}
\forall f, g \in \B, \qquad \| fg \|_{\B} \lesssim \| f \|_{L^\infty} \|g\|_{\B} + \| f \|_{\B} \| g \|_{L^\infty}.
\end{equation*}
\end{cor}

\begin{rmk}
The space $B^0_{\infty, 1}$ is not an algebra. If $f, g \in B^0_{\infty, 1}$, one can use Proposition \ref{p:op} to bound the paraproducts $T_f(g)$ and $T_g(f)$, but not the remainder $\mathcal{R}(f, g)$, because the sum of the regularities of $f$ and $g$ is zero.
\end{rmk}

\subsection{Transport equations and commutator estimates}

In this section, we focus on the transport. We refer again to Chapters 2 and 3 of \cite{BCD} for additional details. We study the initial value problem
\begin{equation}\label{eq:TV}
\begin{cases}
\partial_t f + v \cdot \nabla f = g \\
f_{|t = 0} = f_0\,.
\end{cases}
\end{equation}
The velocity field $v=v(t,x)$ will always assumed to be divergence-free, \tsl{i.e.} $\D(v) = 0$, and Lipschitz. It is therefore practical to make the following definition: the triplet $(s, p, r) \in \mathbb{R} \times [1, +\infty]^2$ will be said to satisfy the \emph{Lipschitz condition} if 
condition \eqref{i_eq:Lip} holds. Notice that this implies the embedding $\B \hookrightarrow W^{1, \infty}$.

Finding \textsl{a priori} estimates for problem \eqref{eq:TV} in Besov spaces requires to look at the dyadic blocks. Let $j \geq -1$. Applying $\Delta_j$ to the transport equation yields
\begin{equation*}
\partial_t \Delta_j f + (v \cdot \nabla) \Delta_j f = \big[ v \cdot \nabla, \Delta_j \big] f + \Delta_j g,
\end{equation*}
where $\big[ v \cdot \nabla, \Delta_j \big] f$ is the commutator $(v \cdot \nabla) \Delta_j - \Delta_j (v \cdot \nabla)$. The following estimate is of recurring use in this article (see Lemma 2.100 and Remark 2.101 in \cite{BCD}). 

\begin{lemma}\label{l:CommBCD}
Assume that $v \in \B$, with $(s, p, r)$ satisfying the Lipschitz condition \eqref{i_eq:Lip}. Then
\begin{equation*}
\forall f \in \B, \qquad  2^{js} \left\| \big[ v \cdot \nabla, \Delta_j \big] f  \right\|_{L^p} \lesssim c_j \Big( \|\nabla v \|_{L^\infty} \| f \|_{\B} + \|\nabla v \|_{B^{s-1}_{p, r}} \|\nabla f \|_{L^\infty} \Big),
\end{equation*}
for some sequence $\big(c_j\big)_{j\geq -1}$ in the unit ball of $\ell^r$. 
\end{lemma}

We also require another commutator lemma (this is Lemma 2.99 in \cite{BCD}).

\begin{lemma}\label{l:ParaComm}
Let $\k$ be a smooth function on $\mathbb{R}^d$, which is homogeneous of degree $m$ away from a neighbourhood of $0$. Then, for a vector field $v$ such that $\nabla v \in L^\infty$, one has:
\begin{equation*}
\forall f \in \B, \qquad \left\| \big[ \mathcal{T}_v, \k(D) \big] f \right\|_{B^{s-m+1}_{p, r}} \lesssim \|\nabla v\|_{L^\infty} \|f\|_{\B}.
\end{equation*}
\end{lemma}

\begin{rmk}
From the proof of this lemma in \cite{BCD}, it appears that the result holds if $\k$ is a bounded homogeneous function of degree zero, as it does not blow up at $\xi = 0$. In particular, if $\k(D) = \mathbb{P}$ is the Leray projector, then we have
\begin{equation*}
\forall f \in \B, \qquad \left\| \big[ \mathcal{T}_v, \mathbb{P} \big] f \right\|_{B^{s+1}_{p, r}} \lesssim \|\nabla v\|_{L^\infty} \|f\|_{\B}.
\end{equation*}
\end{rmk}

All this results in a well-posedness theorem for problem \eqref{eq:TV} in general Besov spaces (see Theorem 3.19 in \cite{BCD}).

\begin{thm}\label{th:transport}
Let $(s, p, r) \in \mathbb{R} \times [1, +\infty]^2$ satisfy the Lipschitz condition \eqref{i_eq:Lip}. Given some $T>0$, let $g \in L^1_T(\B)$ and $v \in L^1_T(\B)$ such that there exist real numbers $q > 1$ and $M > 0$ for which $v \in L^q_T(B^{-M}_{\infty, \infty})$. Finally, let $f_0 \in \B$ be an initial datum. Then, the transport equation \eqref{eq:TV} has a unique solution $f$ in 
\begin{itemize}
\item the space $C^0_T(\B)$, if $r < +\infty$;
\item the space $\left( \bigcap_{s'<s} C^0_T(B^{s'}_{p, \infty}) \right) \cap C^0_{w, T}(B^s_{p, \infty})$, if $r = +\infty$.
\end{itemize}
Moreover, this unique solution satisfies the following estimate:
\begin{equation*} 
\| f \|_{L^\infty_T(\B)} \leq \exp \left( C \int_0^T \| \nabla v \|_{B^{s-1}_{p, r}} \right)
\left\{ \| f_0 \|_{\B} + \int_0^T \exp \left( - C \int_0^\tau \| \nabla v \|_{B^{s-1}_{p, r}} \right) \| g(\tau) \|_{\B} {\rm d} \tau  \right\},
\end{equation*}
for some constant $C = C(d, p, r, s)>0$.
\end{thm}


\end{document}